\documentclass[11pt,a4paper]{amsart}
\usepackage[english]{babel}
\usepackage[utf8]{inputenc}
\pdfoutput=1
  \pdfpageattr{/Group <</S /Transparency /I true /CS /DeviceRGB>>}

\usepackage[colorinlistoftodos,bordercolor=orange,backgroundcolor=orange!20,linecolor=orange,textsize=scriptsize]{todonotes}
\usepackage{amsmath,amsfonts,amssymb,amsthm}
\usepackage{mathtools}
\usepackage{bm}
\usepackage{graphicx}
\usepackage{url}
\usepackage{listings}
\usepackage{color}
\usepackage{courier}
\usepackage{xspace}
\usepackage{caption}
\usepackage{hyperref}
\usepackage{algorithm}
\usepackage{listings}
\usepackage{tikz}
\usetikzlibrary{calc}
\usepackage{csquotes}
\usepackage{booktabs}
\usepackage{tikz-cd}

\newcommand\1{{\mathbf 1}} 
\newcommand{\NN}{\mathbb{N}}
\newcommand{\ZZ}{\mathbb{Z}}
\newcommand{\TT}{\mathbb{T}} 
\newcommand{\QQ}{\mathbb{Q}}
\newcommand{\RR}{\mathbb{R}}
\newcommand{\CC}{\mathbb{C}}
\newcommand{\KK}{\mathbb{K}}
\newcommand{\LL}{\mathbb{L}}
\newcommand{\cD}{\mathcal{D}}
\newcommand{\cR}{\mathcal{R}}
\newcommand{\cP}{\mathcal{P}}
\newcommand{\cK}{\mathcal{K}}
\newcommand\tropvariety[1]{{\mathcal T}#1} 

\newcommand\smallSetOf[2]{\{#1\,|\,#2\}}
\newcommand\bigSetOf[2]{\biggl\{#1\,\biggm|\,#2\biggr\}}
\newcommand\SetOf[2]{\left\{#1\,\vphantom{#2}\right|\left.\vphantom{#1}\,#2\right\}}

\newcommand\puiseux[2]{{#1}\{\!\{{#2}\}\!\}}
\newcommand\stableintersection{\cap_{\rm st}}
\newcommand\lex[1]{\mathbf{#1}}
\newcommand{\mon}{{\Gamma}}

\newcommand\hahnseries[2]{{#1}[\![{#2}]\!]} 
\newcommand\hahnpoly[2]{{#1}[{#2}]} 
\newcommand\hahnfrac[2]{{#1}({#2})} 
\newcommand\hahnconvergent[2]{{#1}\{{#2}\}} 
\DeclareMathOperator\lt{lt} 
\DeclareMathOperator\lc{lc} 
\DeclareMathOperator\supp{supp} 
\DeclareMathOperator\val{val} 
\DeclareMathOperator\trop{trop} 
\DeclareMathOperator\sta{star} 
\DeclareMathOperator\inte{relint} 
\DeclareMathOperator\argmin{argmin} 
\newcommand\smallpmatrix[1]{\left(\begin{smallmatrix} #1 \end{smallmatrix}\right)}
\newcommand\smallcolvectwo[2]{\smallpmatrix{#1 \\ #2}} 

\DeclarePairedDelimiter\abs{\lvert}{\rvert}

\definecolor{appcolor}{rgb}{0,0,0}
\definecolor{commentcolor}{rgb}{.3,.3,.3}
\definecolor{stringcolor}{RGB}{103,5,172}
\definecolor{typecolor}{RGB}{140,75,0}
\definecolor{propcolor}{RGB}{0,100,14}
\definecolor{light_gray}{gray}{0.7}

\lstdefinelanguage{pmshell}
{
  basicstyle=\small\ttfamily,
  keywords=[1]{Hypersurface,ValuatedMatroid,SubdivisionOfPoints},
  morekeywords=[2]{N_ISOLATED, N_FAMILIES, POLYNOMIAL, DEGREE, GENUS, BASES, VALUATION_ON_BASES, WEIGHTS, N_ELEMENTS, FAR_VERTICES, CONES, POINTS, VERTICES, TIGHT_SPAN, MAXIMAL_POLYTOPES, MONOMIALS, COEFFICIENTS, MAXIMAL_CELLS, RAYS, N_RAYS, N_MAXIMAL_CELLS, FACETS, LINEAR_SPAN, CURVE_EDGE_LENGTHS, VISUAL},
  morekeywords=[3]{map, grep, join, secondary_cone},
  stringstyle=\color{stringcolor},
  keywordstyle=[1]\color{typecolor},
  keywordstyle=[2]\color{propcolor},
  keywordstyle=[3]\color{typecolor},
  numbers=none,
  captionpos=b,
  showspaces=false,
  showstringspaces=false,
  morestring=[b]",
  escapechar=&,
  frame=single,
}

\theoremstyle{definition}
\newtheorem{definition}{Definition}[section]
\newtheorem{example}[definition]{Example}

\newtheorem{remark}[definition]{Remark}

\theoremstyle{plain}
\newtheorem{theorem}[definition]{Theorem}
\newtheorem{corollary}[definition]{Corollary}
\newtheorem{lemma}[definition]{Lemma}
\newtheorem{proposition}[definition]{Proposition}

\newtheorem{question}[definition]{Question}

\title[Tropical Geometry of Higher Rank]{Convergent Hahn Series and\\ Tropical Geometry of Higher Rank}

\author{Michael Joswig}
\address{Technische Universit\"at Berlin, Chair of Discrete Mathematics / Geometry, Berlin \& MPI Mathematics in the Sciences, Leipzig, Germany}
\curraddr{}
\email{joswig@math.tu-berlin.de}
\thanks{Research by the first author is partially supported by Einstein Stiftung Berlin and Deutsche Forschungsgemeinschaft (EXC 2046: \enquote{MATH$^+$}, SFB-TRR 195: \enquote{Symbolic Tools in Mathematics and their Application}, and GRK 2434: \enquote{Facets of Complexity}).  Additional support by Institut Mittag-Leffler within the program \enquote{Tropical Geometry, Amoebas and Polytopes} is gratefully acknowledged.}

\author{Ben Smith}
\address{University of Manchester, Oxford Road, Manchester, M13 9PL, United Kingdom}
\email{benjamin.smith-3@manchester.ac.uk}
\thanks{The second author was supported by the EPSRC (1673882), the Heilbronn Institute for Mathematical Research and was funded by the \enquote{Eileen Coyler Prize} from Queen Mary University of London to visit the first author.}

\subjclass[2020]{14T05 (12D15, 52B11)}

\begin{document}

\begin{abstract}
  We propose to study the tropical geometry specifically arising from convergent Hahn series in multiple indeterminates.
  One application is a new view on stable intersections of tropical hypersurfaces.
  Another one is perturbations of rank one tropical polytopes, which is beneficial for algorithmic purposes.
\end{abstract}

\maketitle

\section{Introduction}
\noindent
Tropical geometry connects algebraic geometry over some valued field $\KK$ with polyhedral geometry over the semifield $\TT=(\RR,\min,+)$.
Often it is less important which field $\KK$ is chosen, and a common choice is the field $\puiseux{\CC}{t}$ of formal Puiseux series with complex coefficients.
By taking the convergence of series in $\puiseux{\CC}{t}$ into account, we can pull back the valuation map $\val:\puiseux{\CC}{t}\to\TT$ and then substitute $t$ by some complex number.
Diagrammatically this can be written as
\begin{equation}\label{eq:transfer}
  \begin{tikzcd}
    \TT & \arrow{l}{\val} \puiseux{\CC}{t} \arrow[dashed]{r} & \CC \enspace .
  \end{tikzcd}
\end{equation}
Notice that the substitution, which is represented by the dashed arrow, depends on the choice of the complex number substituted.
This number must lie within the radius of convergence, and so the dashed arrow is \emph{not} a map defined for all Puiseux series.
Nonetheless, conceptually this opens up a road for transferring metric information from tropical geometry over $\TT$ via algebraic geometry over $\puiseux{\CC}{t}$ to metric information over $\CC$.
This idea was exploited recently to obtain new and surprising complexity results for ordinary linear optimisation \cite{ABGJ:2014}, \cite{ABGJ:2018}.
The purpose of this article is to explore generalisations of this concept to tropical geometry of higher rank and its applications.
Observe that a diagram like \eqref{eq:transfer} does not make any sense for an arbitrary valued field: in general, there is no map equivalent to the substitution of $t$ by a complex number.
Instead of Puiseux series in this paper we prefer to work with the larger field of Hahn series as there the valuation map is onto the reals.

Tropical geometry of higher rank was pioneered in articles by Aroca \cite{Aroca:2010a}, \cite{Aroca:2010b} and Aroca, Garay and Toghani \cite{ArocaGarayToghani:2016}.
Their work is motivated by research on algebraic ways of solving systems of differential equations.
This gives a natural notion of a tropical hypersurface of higher rank, and this allows for a higher rank version \cite[Theorem 8.1]{Aroca:2010a} of Kapranov's fundamental theorem of tropical geometry \cite[Theorem 3.2.5]{Tropical+Book}.
Banerjee \cite{Banerjee:2015} focused on tropicalisations of closed subschemes of the torus over higher dimensional local fields.
Foster and Ranganathan \cite{FosterRanganathan:2016degeneration,FosterRanganathan:2016hahn} later considered a more general notion of tropicalisation, and proved these tropicalisations were connected using methods from analytic geometry.
The main result of \cite{FosterRanganathan:2016degeneration} is a generalisation of a result of Gubler \cite{Gubler:2013} to higher rank.
While the exposition in \cite{Banerjee:2015} is restricted to higher rank valuations which are discrete, the articles \cite{FosterRanganathan:2016degeneration} and \cite{FosterRanganathan:2016hahn} also cover the non-discrete case.
Since the above work with more general local fields, a diagram like \eqref{eq:transfer} does not occur.

One approach to tropical geometry goes through the process of tropicalising classical algebraic varieties.
Here we consider a variety $V$ over some valued field $\KK$, and the \emph{tropicalisation} of $V$ is obtained by applying the valuation map to each point of $V$ coordinatewise.
The fundamental theorem of tropical geometry says that this agrees with intrinsic ways to describe a tropical variety \cite[Theorem 3.2.5]{Tropical+Book}.
While typically $\KK$ is assumed to be algebraically closed, a closer view shows that it is worthwhile to also consider real-closed fields, and this leads to Alessandrini's work on the tropicalisation of semialgebraic sets \cite{Alessandrini:2013}.
Working over an ordered field has the advantage that the cancellation of terms, which is the source of many technical challenges in tropical geometry, can be controlled via keeping track of the signs.
This is essential for applications to optimisation as in \cite{ABGJ:2014}, \cite{ABGJ:2018}, \cite{AllamigeonGaubertSkomra:1610.06746} and \cite{ETC}.

Digging even deeper, it turns out that tropicalising with respect to specially crafted fields can allow for stronger results in applications of tropical geometry.
For instance, \cite[Theorem~4.3]{ABGJ:2014}, which is about the complexity of the simplex method, hinges on employing convergent real Hahn series of higher rank; cf.\ \cite[Theorem~3.12]{ABGJ:2014}.
Despite the fact that the basic idea is simple, the algebraic, topological and analytic properties are somewhat subtle.
This is our point of departure, and in Section~\ref{sec:convergent+hahn} we begin with a general description of fields of convergent Hahn series in more than one indeterminate.
A first observation is Proposition~\ref{prop:substitution} on partial evaluations of convergent Hahn series of higher rank, and this gives rise to a higher rank analogue \eqref{eq:frac+diagram} of \eqref{eq:transfer}.
Interestingly, at this level of detail, it is natural to first study Hahn series with real coefficients (leading to ordered and real-closed fields) before addressing the complex (and algebraically closed) case.

With this we are prepared for the main part of this paper, on tropical hypersurfaces of higher rank, which is Section~\ref{sec:hypersurfaces}.
For conciseness we restrict our attention to rank two; yet all statements admit straightforward generalisations to arbitrarily high rank.
Our first contributions are Theorem~\ref{thm:union+polyhedra} and its Corollary~\ref{cor:hypersurface+closure} which describes the closure in the Euclidean topology of an arbitrary rank two tropical hypersurface in terms of ordinary polyhedra.
These results require us to study sets defined by finitely many linear inequalities with respect to the lexicographic ordering on the semimodule $(\TT_2)^d$, which we call \emph{lex-polyhedra}.
A key ingredient in the analysis is the diagram \eqref{eq:poly+diagram} which is a consequence of~\eqref{eq:frac+diagram}.

In tropical geometry, it is fundamental that intersections of tropical varieties do not need to be tropical varieties, in general.
This fact gives rise to technical challenges in proofs in tropical geometry, and the concept of \emph{stable intersection} frequently offers a path towards a solution \cite[\S3.6]{Tropical+Book}.
This is the topic of Section~\ref{sec:stable}.
Theorem~\ref{thm:stable+intersection}, which is a consequence of our main result, allows us to view stable intersection as an instance of the \enquote{symbolic perturbation}
paradigm from computational geometry; e.g., see \cite{Edelsbrunner87} and \cite{EmirisCannySeidel97}.
This should be compared with \cite[\S3.2]{ABGJ:2014} and \cite[\S5]{AllamigeonKatz:1408.6176}, where a similar idea has been applied to obtain perturbations of rank one tropical linear programs; or with the approach to \enquote{genericity by deformation} of monomial ideals \cite[\S6.3]{MillerSturmfels:2005}.
We also motivate the Euclidean topology as a valuable tool for higher rank tropical geometry, as taking Euclidean closures allows us to obtain Theorem~\ref{thm:stable+intersection} far easier.

In Section~\ref{sec:convex} we follow a completely different strand in tropical geometry.
This is about $(\min,+)$-linear algebra, which has been studied for several decades with numerous applications in optimisation, discrete event systems and other areas; cf.\ \cite{BacelliCohenOlsderQuadrat92}, \cite{Butkovic10}, \cite{ETC} and the references there.
Like all of tropical geometry, this has a specifically polyhedral geometry flair; Develin and Yu \cite[Proposition 2.1]{DevelinYu07} proved that the tropical cones (which are precisely the $(\min,+)$-semimodules) agree with the images of ordinary cones over real Hahn series under the valuation map.
This can be seen as a version of the fundamental theorem for tropical convexity.
Working over real Hahn series which are convergent allows us to relate three kinds of objects: ordinary cones over real Hahn series, tropical cones and ordinary cones over the reals.
This is expressed in \eqref{eq:transfer}, and this is the crucial idea behind the recent complexity results on ordinary linear and semidefinite programming via tropical geometry \cite{ABGJ:2014}, \cite{ABGJ:2018}, \cite{AllamigeonGaubertSkomra:1610.06746}; cf.\ Remark~\ref{rem:m-zero}.
Proposition~\ref{prop:cone} is a version of the Develin--Yu Theorem for convergent Hahn series of rank two.
Yet, the core of this section are Theorems~\ref{thm:lex-complex} and~\ref{thm:cone+structure}.
The former gives a decomposition for rank two tropical cones analogous to the covector decomposition for rank one tropical cones \cite[\S6.3]{ETC}; the latter is a tropical convexity analogue to our Theorem~\ref{thm:union+polyhedra} on rank two tropical hypersurfaces.

Section~\ref{sec:remarks} ends this article with several concluding remarks.
In particular, we hint at generalising our results from rank two to arbitrary rank.

\section*{acknowledgements}
  We are indebted to Alex Fink, Tyler Foster, Jeff Giansiracusa, Georg Loho, Diane Maclagan, Kalina Mincheva, Dhruv Ranganathan, Claus Scheiderer, and Sascha Timme for fruitful discussions and valuable hints.
  Further, we are indebted to three anonymous referees for their comments and suggestions.

\section{Higher rank valued fields}
\noindent
We begin by recalling of formal Hahn series and their convergence.
For more details, we refer to~\cite{vanderHoeven:2006} and \cite{MullerStrohmaier:2014}.

\subsection{Multivariate Hahn series}\label{sec:formal+hahn}
\noindent
Let $(\mon,+)$ be a totally ordered abelian group, and let $\cR$ be a commutative ring with $1$.
A formal series
\[
  \gamma \ = \ \gamma(T) \ = \ \sum_{\alpha \in \mon} c_\alpha \cdot T^{\alpha} \quad \text{with } c_\alpha\in\cR
\]
is called a \emph{(formal) Hahn series} if the \emph{support} $\supp(\gamma)=\smallSetOf{\alpha \in \mon }{ c_\alpha \neq 0 }$ is well-ordered.
We write $\hahnseries{\RR}{T^\mon}$ for the set of Hahn series.
With coefficient-wise addition and the usual convolution product, Hahn series form a commutative ring, which contains $\cR$ as $\cR\cdot T^0$, the subring of constant Hahn series.
If $\cR$ is a field, then so is $\hahnseries{\cR}{T^\mon}$.
We are particularly interested in the case where $\cR=\RR$ and $\Gamma=\RR^m$, equipped with the lexicographic ordering, and we abbreviate $\hahnseries{\RR}{T}=\hahnseries{\RR}{T^{\RR^m}}$.
\begin{remark}
  It is known that each well-ordered subset of $\RR^m$ is at most countable.
  We owe the following sketch of a proof to an anonymous referee.
  Assume that $W$ is a well-ordered subset of $\RR$ which is uncountable.
  The uncountably many half-open intervals $[w,w')$ are pairwise disjoint, where $w'$ is the successor of $w$ in $W$.
  As each such interval contains rational numbers, this contradicts the countability of the rationals.
  This argument can be extended to $\RR^m=\RR\times\RR^{m-1}$ for $m>1$ by considering the projections onto the two factors and applying induction.
\end{remark}

As the support of a Hahn series $\gamma\in\hahnseries{\RR}{T}$ is a well ordered set, the \emph{order}
\[
  \val(\gamma) \ := \ \min \supp(\gamma)
\]
of $\gamma$ is defined, unless $\gamma=0$.
If $\val(\gamma)$ is $\alpha_0$, the \emph{leading term} $\lt(\gamma)$ is $c_{\alpha_0} T^{\alpha_0}$, and the \emph{leading coefficient} $\lc(\gamma)$ is $c_{\alpha_0}$.
A nonzero Hahn series is \emph{positive} if $c_{\alpha_0}$ is positive.
This definition turns $\hahnseries{\RR}{T}$ into an ordered field.
In fact, since the additive group of $\RR^m$ is divisible, the field $\hahnseries{\RR}{T}=\hahnseries{\RR}{T^{\RR^m}}$ is real-closed; see \cite[\S4]{Hils:2018} and \cite[\S1.2]{Bochnak:1998}.

The number $m$ is the rank of $\Gamma=\RR^m$ as a free abelian group.
Therefore we say that $\hahnseries{\RR}{T}$ is a field of Hahn series of \emph{rank} $m$.
We call the triplet $\TT_m:=(\RR^m,\min,+)$, where $\min$ is the minimum with respect to the lexicographic ordering, the \emph{rank $m$ tropical semifield}.
The order map gives $\hahnseries{\RR}{T}$ the structure of a valued field with valuation
\[
\val\colon \hahnseries{\RR}{T} \rightarrow \TT_m \cup \{\infty\} \enspace ,
\]
and value group $\RR^m$.
Furthermore, restricting the order map to positive Hahn series gives a homomorphism
\[
  \val\colon\hahnseries{\RR}{T}_{>0}\to\TT_m
\]
of semirings, which reverses the ordering; i.e., $\gamma\leq\gamma'$ implies $\val(\gamma)\geq\val(\gamma')$.

The field of formal Hahn series $\hahnseries{\RR}{T}$ is a large field that satisfies many desirable properties, in particular its real-closedness.
This entails that $\hahnseries{\CC}{T}=\hahnseries{\CC}{T^{\RR^m}}$ is an algebraically closed valued field of characteristic zero.
For $i$ an imaginary unit satisfying $i^2=-1$, we have $\hahnseries{\CC}{T}=\hahnseries{\RR}{T}+i\hahnseries{\RR}{T}$.
This makes it a natural candidate for tropical geometry.
Furthermore, the valuation map being surjective will be an invaluable property when discussing higher rank tropical objects.
Yet, occasionally, we will require mild assumptions concerning convergence, beyond just formal summations.
Therefore, we will treat the field of formal Hahn series as an \enquote{umbrella} field, and consider suitable subfields, which we discuss next.

\subsection{Convergent Hahn series}\label{sec:convergent+hahn}
\noindent
Consider the field $\hahnseries{\RR}{T}$ of Hahn series of rank $m$ with real coefficients.
We may view $T$ as a tuple of $m$ indeterminates $(t_1,\dots,t_m)$ and rewrite the formal monomial $T^\alpha$ as $t_1^{\alpha_1}\cdots t_m^{\alpha_m}$.

We say that a Hahn series $\gamma \in \hahnseries{\RR}{T}$ is \emph{convergent} if there exists a vector $r = (r_1,\dots, r_m)$ of positive real numbers such that the real series
\[
\gamma(\rho) \ = \ \sum_\alpha c_\alpha \rho_1^{\alpha_1} \cdots \rho_m^{\alpha_m}
\]
obtained by substituting $T$ converges absolutely for all $\rho \in (0,r_1] \times \cdots (0,r_m]$.
We call $r$ a \emph{polyradius} for $\gamma$.
The map $\rho\mapsto\gamma(\rho)$ is continuous and real analytic on the interior of its domain of convergence.

If some Hahn series is convergent, we can additionally consider partial substitutions.
Let us consider a second tuple $U=(u_1,\dots,u_n)$ of $n$ indeterminates. 
Extending the construction from Section~\ref{sec:formal+hahn}, we arrive at the field of Hahn series $\hahnseries{\RR}{T,U}=\hahnseries{\RR}{T^{\RR^m},U^{\RR^n}}$ of rank $m+n$.
As a valued field, the value group of $\hahnseries{\RR}{T,U}$ is $\RR^{m+n}$ with lexicographical ordering, therefore we can consider the indeterminates $U$ as having smaller valuation than $T$.
Note that both $\hahnseries{\RR}{T}$ and $\hahnseries{\RR}{U}$ are naturally subfields.

Let $\gamma(T,U) \in \hahnseries{\RR}{T,U}$ be convergent for some polyradius $(r,s)$, and let $\sigma$ be a vector of positive reals in $(0,s_1] \times \cdots \times (0,s_n]$.
Then we can also consider the (partial) evaluation of $\gamma$ by $\sigma$ by substituting $U$ for $\sigma$:
\begin{equation}\label{eq:gamma-T-sigma}
  \gamma(T,\sigma) \ = \ \sum_{(\alpha,\beta)} c_{\alpha,\beta} T^\alpha \sigma^\beta \ = \ \sum_{\alpha} \biggl({\sum_{\beta} c_{\alpha,\beta} \sigma_1^{\beta_1} \cdots \sigma_{n}^{\beta_{n}}}\biggr) T^\alpha \, .
\end{equation}

\begin{proposition}\label{prop:substitution}
  Let  $\gamma(T,U)\in\hahnseries{\RR}{T,U}$ be a Hahn series which converges in the polyradius $(r,s)=(r_1,\dots,r_m,s_1,\dots,s_n)$.
  Then the partial evaluations of $U=(u_1,\dots,u_n)$ at constants $\sigma=(\sigma_1,\dots,\sigma_n)$ with $\sigma_i\in(0,s_i]$ yields a convergent Hahn series $\gamma(T,\sigma)\in\hahnseries{\RR}{T}$.

  A similar result holds for the partial evaluations of $T=(t_1,\dots,t_m)$.
\end{proposition}
\begin{proof}
One can group the terms of $\gamma$ in the following way:
  \begin{equation}\label{eq:substitution:formal}
    \gamma(T,U) \ = \ \sum_{(\alpha,\beta)} c_{\alpha,\beta} T^\alpha U^\beta \ = \ \sum_{\alpha} \biggl(\underbrace{\sum_{\beta} c_{\alpha,\beta} U^\beta}_*\biggr) T^\alpha 
  \end{equation}
  This holds formally in $\hahnseries{\RR}{T,U}$ without considering aspects of convergence.
  The support of * in \eqref{eq:substitution:formal} is well-ordered; so $\hahnseries{\RR}{T,U}$ is a subfield of $\hahnseries{(\hahnseries{\RR}{U})}{T}$.
  It remains to show that the series $\sum_{\beta} c_{\alpha,\beta} U^\beta$ converges absolutely in the polyradius $s$ for all $\alpha$.
  For any fixed $\alpha_0$, we get
  \[
    r^{\alpha_0} \sum_{\beta} \abs{c_{\alpha_0,\beta}} s^\beta \ = \ \sum_{\beta} \abs{c_{\alpha_0,\beta}} r^{\alpha_0} s^{\beta} \ \leq \ \sum_{\alpha,\beta} \abs{c_{\alpha,\beta}} r^\alpha s^\beta \ < \ \infty \enspace .
  \]
  The term $r^{\alpha_0}$ does not vanish, and hence $\sum_{\beta} \abs{c_{\alpha_0,\beta}} s^\beta$ is finite.
  Therefore $\gamma(T,\sigma)$ is an element of $\hahnseries{\RR}{T}$, moreover it is convergent in the polyradius $r$.
  The roles of $T$ and $U$ can be exchanged.
\end{proof}
  
The valuation of $\gamma(T,\sigma)$ in \eqref{eq:gamma-T-sigma} is given by
\begin{equation}\label{eq:substitution-val}
  \val(\gamma(T,\sigma)) \ = \ \min\bigSetOf{\alpha}{\sum_{\beta} c_{\alpha,\beta}\sigma^\beta \neq 0} \enspace ,
\end{equation}
and therefore depends on the choice of $\sigma$.
The function $\kappa_\alpha$ which sends $\sigma$ to $\sum_{\beta} c_{\alpha,\beta}\sigma^\beta$ depends on $\alpha$, and it is real-analytic on the set $(0,s_1) \times \cdots \times (0,s_n)$.
We call $\sigma$ \emph{admissible} for $\gamma$ if $\kappa_\alpha(\sigma)\neq 0$ for all $\alpha$ in the $T$-support of $\gamma$.
In this case the expression \eqref{eq:substitution-val} does not depend on~$\sigma$.

\begin{definition}
  Let $\hahnseries{\RR}{T}$ be the field Hahn series of rank $m$ with real coefficients.
  We call a subring of $\hahnseries{\RR}{T}$ all of whose elements are convergent a \emph{convergent} subring of Hahn series of rank $m$; often we write $\hahnconvergent{\RR}{T}$ for such a subring.
\end{definition}
By construction convergent subrings of $\hahnseries{\RR}{T}$ inherit the same ordering and valuation.
In general, they may be fields, but they do not need to be.
And if they form fields, they may or may not be real-closed.
However, they do have well defined (partial) evaluation maps for small admissible values.
As with formal Hahn series, restricting to only positive elements makes the valuation map an order-reversing homomorphism.
For a pair of convergent subrings $\hahnconvergent{\RR}{T}$, $\hahnconvergent{\RR}{T,U}$ we obtain the following diagram of semirings:
\begin{equation}\label{eq:frac+diagram}
  \begin{tikzcd}
    \hahnconvergent{\RR}{T}_{>0} \arrow{d}{\val_m} \arrow{r}{\iota} &  \hahnconvergent{\RR}{T,U}_{>0} \arrow{d}{\val_{m+n}} \arrow[dashed]{r}{\pi_{u\mapsto\sigma}} & \hahnconvergent{\RR}{T}_{>0} \arrow{d}{\val_m} \\
    \TT_m \arrow{r}{\iota_*} & \TT_{m+n} \arrow{r}{\pi_{u*}} & \TT_m
  \end{tikzcd}
\end{equation}
A few remarks are in order.
Whenever we wish to distinguish between the various valuation maps we add the appropriate index to the symbol \enquote{$\val$}.
We assume that the map $\iota:\hahnconvergent{\RR}{T}\to\hahnconvergent{\RR}{T,U}$ is an embeddings of subrings such that the induced map $\iota_*:\TT_m\to\TT_{m+n}$ sends the exponent $\alpha$ to $(\alpha,0)$.
The map $\pi_{u*}$ is the projection $(\alpha,\alpha')\mapsto\alpha$ onto the first coordinate.

The dashed arrow labelled $\pi_{u\mapsto\sigma}$ in the diagram \eqref{eq:frac+diagram} is a subtle point.
We define $\pi_{u\mapsto\sigma}(\delta(T,U))$ to be the partial evaluation $\delta(T,\sigma)$ and assume that this is contained in $\hahnconvergent{\RR}{T}$.
The latter expression depends on $\sigma$ (and its admissibility), and hence that map is only partial.
However, for each $\delta\in\hahnconvergent{\RR}{T,U}$ a set of admissible values is defined, and the order of the resulting element in $\hahnconvergent{\RR}{T}$ does not depend on the specific choice of $\sigma$.
In this sense we assume that the diagram \eqref{eq:frac+diagram} commutes, despite the fact that $\pi_{u\mapsto\sigma}$ is not globally defined.

We need to clarify that convergent subrings of $\hahnseries{\RR}{T}$ and $\hahnseries{\RR}{T,U}$ exist which allow for the diagram \eqref{eq:frac+diagram} to commute.
In fact, there is a wide variety of choices; see~\cite{vanderHoeven:2006}.
However, some constructions are fairly involved, and here we are less interested in the specific arithmetic or analytic properties.
For the most part we are content with the following simple example.
\begin{example}
  We call a Hahn series of rank $m$ with finite support an \emph{$m$-variate Hahn polynomial}, and this is always convergent.
  The Hahn polynomials $\hahnpoly{\RR}{T}$ form a convergent subring of $\hahnseries{\RR}{T}$, and this leads to a commutative diagram like \eqref{eq:frac+diagram}.
  Sometimes it is more convenient to work with a field; in that case we can pass to the quotient field of the Hahn polynomials.
  Expanding the inverse of a Hahn polynomial via the geometric series yields a series which is again convergent.
  We call that quotient the field of \emph{$m$-variate Hahn fractions}, and we denote it $\hahnfrac{\RR}{T}$.
  The Hahn fractions form an ordered field which is not real-closed.
\end{example}
Note that for both $\hahnpoly{\RR}{T}$ and $\hahnfrac{\RR}{T}$, the order map $\val_m$ is surjective onto $\RR^m$.
While this is not strictly necessary it makes it easier to formulate some results below.
For example, one can consider the subfield of Hahn fractions with rational coefficients and exponents, which is countable.
For that subfield the order map is clearly not surjective onto $\RR^m$; but that field is well suited algorithmically.
In the univariate case this construction recovers the \enquote{Puiseux fractions} from \cite{JoswigLohoLorenzSchroeter:2016} and \cite[\S2.6]{ETC}.

\begin{example}\label{exmp:diagram}
  Consider the case $m=n=1$ with $T=(t_1)$, $t_1=t$ and $U=(u_1)$, $u_1=u$.
  Let us look at the series
  \[
    \begin{split}
      \gamma(t,u) \ &= \ \sum_{\alpha\in\NN,\, \beta\in\NN\setminus\{0\}} t^\alpha u^\beta  \ = \ \sum_{\alpha\in\NN} \biggl( \sum_{\beta\in\NN\setminus\{0\}} u^\beta \biggr) t^\alpha \\
      &= \  \biggl( \sum_{\alpha\in\NN} t^\alpha \biggr)  \biggl( \sum_{\beta\in\NN\setminus\{0\}} u^\beta \biggr) \ = \ \biggl( \sum_{\alpha\in\NN} t^\alpha \biggr)  \biggl( u \cdot \sum_{\beta\in\NN} u^\beta \biggr)
      \enspace \\
      &= \frac{1}{1-t} \cdot \frac{u}{1-u} \enspace .
    \end{split}
  \]
  While this is not a Hahn polynomial itself, it is a positive element in the field of Hahn fractions $\hahnfrac{\RR}{t,u}$.
  For the polyradius of convergence we may pick, e.g., $(\tfrac{3}{4},\tfrac{3}{4})$.

  The partial evaluation $u\mapsto\tfrac{1}{2}$ is defined, and we arrive at
  \[
    \pi_{u\mapsto\frac{1}{2}}(\gamma(t,u))  \ = \ \gamma(t,\tfrac{1}{2}) \ = \ \frac{1}{1-t} \cdot \tfrac{1}{2} \frac{1}{1-\tfrac{1}{2}} \ = \ \frac{1}{1-t} \enspace ,
  \]
  which is an element of $\hahnfrac{\RR}{t}$.
  Clearly, other partial evaluations yield other results, such as, e.g.,
  \[
    \pi_{u\mapsto\frac{1}{3}}(\gamma(t,u)) \ = \ \gamma(t,\tfrac{1}{3}) \ = \ \frac{1}{1-t} \cdot \tfrac{1}{3} \frac{1}{1-\tfrac{1}{3}} \ = \ \tfrac{1}{2}\frac{1}{1-t} \enspace .
  \]
  We have $\val_2(\gamma)=(0,1)$ and
  \[
    \val\bigl(\gamma(t,\tfrac{1}{2})\bigr) \ = \ \val\bigl(\gamma(t,\tfrac{1}{3})\bigr) \ = \ 0 \ = \ \pi_{u*}(\val_2(\gamma)) \enspace .
  \]
  In this example all real numbers in the open interval $(0,1)$ are admissible.
\end{example}

In Proposition~\ref{prop:substitution} the roles of the $T$-variables and the $U$-variables are symmetric.
Yet the definition of $\val_2$ breaks this symmetry.
The following example shows that $T$ and $U$ cannot be exchanged in \eqref{eq:frac+diagram}.
Nonetheless the notations \enquote{$\pi_{t\mapsto\rho}$} and \enquote{$\pi_{t*}$} make sense; the map $\pi_{t*}$ is the projection $(\alpha,\alpha')\mapsto\alpha'$ onto the second coordinate.

\begin{example}\label{exmp:not-interchangeable}
  For $\gamma(t,u) = tu^3+t^2u^{-1}$ in $\hahnfrac{\RR}{t,u}$ we have $\val_2(\gamma)=(1,3)$.
  According to \eqref{eq:frac+diagram} we have the equality
  \[
    \val(\pi_{u \mapsto 1}(\gamma)) \ = \ \val(t+t^2) \ = \ 1 \ = \ \pi_{u*}(1,3) \enspace .
  \]
  Yet, here the roles of $t$ and $u$ cannot be exchanged:
  \[
    \val(\pi_{t \mapsto 1}(\gamma)) \ = \ \val(u^{-1}+u^3) \ = \ -1 \ \neq \ \pi_{t*}(1,3) \enspace .
  \]
\end{example}

\begin{remark}\label{rem:m-zero}
 It is worth noting that the case $m=0$ and $n=1$ does make sense in \eqref{eq:frac+diagram}.
 Then we have $T=()$ and $U=(u)$, leading to $\hahnconvergent{\RR}{T}\cong\RR$ and $\TT_0=\{0\}$; the map $\iota$ sends $c\in\RR_{>0}$ to the constant Hahn series $c\cdot u^0\in \hahnconvergent{\RR}{u}$, and $\val_0$ is the trivial valuation on the positive reals.
 The right half of the diagram now degenerates to the real version of \eqref{eq:transfer} as:
 \begin{equation}
   \begin{tikzcd}
     \hahnconvergent{\RR}{u}_{>0} \arrow{d}{\val} \arrow[dashed]{r}{\pi_u} & \RR_{>0} \\
     \TT &
   \end{tikzcd}
 \end{equation}
 In fact, this can be exploited to pull back metric information from the semimodule $\TT^k$ and project it to (the positive orthant of) the real vector space~$\RR^k$, for arbitrary $k$.
 This is a key idea behind \cite{ABGJ:2018}, where this approach was used to show that standard versions of the interior point method cannot solve ordinary linear programs in strongly polynomial time.
\end{remark}

\section{Rank two tropical hypersurfaces}
\label{sec:hypersurfaces}
\noindent
In the sequel we will be investigating the special case where $m = n = 1$, and we postpone questions of convergence.
Moreover, we need an algebraically closed field.
So we consider the field of formal Hahn series of rank two with complex coefficients
\[
  \hahnseries{\CC}{t,u} \ = \ \hahnseries{\RR}{t,u} + i \, \hahnseries{\RR}{t,u} \enspace ,
\]
where $i=\sqrt{-1}$ is an imaginary unit, and this field is equipped with the surjective rank two valuation map $\val_2$.
For improved readability we abbreviate $\LL=\hahnseries{\CC}{t,u}$.

\begin{remark}
  The objects in the following may have two topologies placed on them, the Euclidean topology and the order topology.
  To distinguish between them, we use $\RR^m$ when the underlying set is equipped with the Euclidean topology, and $\TT_m$ when the underlying set is equipped with the order topology.
  Note that $\RR^m$ and $\TT_m$ agree as sets, however it will be useful throughout to differentiate between their topologies. 
\end{remark}

The following is based on \cite{Aroca:2010b} and \cite{Aroca:2010a}.
Given a Laurent polynomial $f = \sum \gamma_sx^s \in \LL[x_1^{\pm},\dots,x_d^{\pm}]$, the \emph{rank two tropicalisation of $f$} is the tropical polynomial obtained from $f$ by applying $\val_2$ to each coefficient and replacing addition and multiplication with their tropical counterparts.
This induces the tropical polynomial map
\begin{alignat*}{2}
  \trop_2(f) : {(\TT_2)}^d &\longrightarrow \TT_2 \\
  p &\longmapsto \min \SetOf{\val_2 (\gamma_s) + \langle s , p \rangle}{s\in\supp(f)} \enspace ,
\end{alignat*}
where $\langle s , p \rangle$ is the pairing
\begin{equation}\label{eq:pairing}
  \begin{aligned}
    \langle -,-\rangle \,:\, \ZZ^d \times {(\TT_2)}^d \ &\longrightarrow \ \TT_2 \\ 
    \bigl( (s_1,\dots,s_d),(p_1,\dots,p_d)\bigr) \ &\longmapsto \ \sum_{i=1}^d (s_ip_{1i}, s_ip_{2i}) \enspace . 
  \end{aligned}
\end{equation}
For every $p \in (\TT_2)^d$ there exists at least one term of the polynomial where $\trop(f)$ attains its minimum, and hence the set 
\[
  \cD_p(f) \ = \ \SetOf{s \in \ZZ^d}{\trop_2(f)(p) = \val_2 (\gamma_s) + \langle s , p\rangle}
\]
is not empty.
\begin{definition}\label{def:hypersurface}
  The \emph{rank two tropical hypersurface} of $f$ is the set
  \[
    \tropvariety_2(f) \ = \ \SetOf{p \in {(\TT_2)}^d}{|\cD_p(f)| > 1} \enspace .
  \]
\end{definition}
As with rank one tropical hypersurfaces, this construction commutes with taking the coordinatewise valuation of the zero set of $f$.
Here it is essential that $\LL$ is algebraically closed and that the valuation map is surjective onto~$\TT_2$.
\begin{theorem}[{\cite[Theorem 8.1]{Aroca:2010a}}] \label{thm:rank2+kapranov}
  Let $f \in \LL[x_1^{\pm},\dots,x_d^{\pm}]$.
  The rank two tropical hypersurface of $f$ is the set of pointwise valuations of the zero set of~$f$, i.e.,
  \[
    \tropvariety_2(f) \ = \ \SetOf{\bigl(\val_2(p_1),\dots,\val_2(p_d)\bigr)}{p\in \LL^d ,\, f(p) = 0} \enspace .
  \]
\end{theorem}

As rank one tropical hypersurfaces are ordinary polyhedral complexes, we would like an analogous structure for rank two tropical hypersurfaces.
As sets $\TT_2$ and $\RR^{2}$ are equal, but the order topology (on $\TT_2$) is strictly finer than the Euclidean topology (on $\RR^2$); recall that the open intervals form a basis of the order topology.
Similarly $(\TT_2)^d$ and $(\RR^{2})^d$ are equal as sets but the respective product topologies are distinct.
In particular, $(\RR^{2})^d$ is homeomorphic with $\RR^{2\times d}$, and we use the latter notation for readability.
Furthermore, we shall write point coordinates as $2{\times}d$-matrices 
\[
\begin{pmatrix}
p_{11} & \dots & p_{1d} \\
p_{21} & \dots & p_{2d}
\end{pmatrix}
\]
to emphasise that points are $d$-tuples of elements of $\RR^2$ or $\TT_2$.

\begin{example} \label{ex:rank+two+hypersurface}
  For the bivariate linear polynomial $f = x_1 + tx_2 + t^2u \in \LL[x_1,x_2]$ its rank two tropical hypersurface is the following subset of $(\TT_2)^2$.
  \[
    \begin{aligned}
      \tropvariety_2(f)
      \ = \ &\SetOf{\smallpmatrix{p_{11} & p_{12} \\ p_{21} & p_{22}}}{\smallcolvectwo{0}{0} + \smallcolvectwo{p_{11}}{p_{21}} = 
        \smallcolvectwo{2}{1} \leq
        \smallcolvectwo{1}{0} + \smallcolvectwo{p_{12}}{p_{22}}} \\
      &\cup \SetOf{\smallpmatrix{p_{11} & p_{12} \\p_{21} & p_{22}}}
      {\smallcolvectwo{1}{0} + \smallcolvectwo{p_{12}}{p_{22}} =
        \smallcolvectwo{2}{1} \leq
        \smallcolvectwo{0}{0} + \smallcolvectwo{p_{11}}{p_{21}}} \\
      &\cup \SetOf{\smallpmatrix{p_{11} & p_{12} \\p_{21} & p_{22}}}
      {\smallcolvectwo{0}{0} + \smallcolvectwo{p_{11}}{p_{21}} =
        \smallcolvectwo{1}{0} + \smallcolvectwo{p_{12}}{p_{22}} \leq
        \smallcolvectwo{2}{1}} \\
    = \ &\SetOf{\smallpmatrix{2 & 1 \\1 & 1} + \smallpmatrix{0 & \lambda_1 \\0 & \lambda_2}}
      {\smallcolvectwo{\lambda_1}{\lambda_2} \geq \smallcolvectwo{0}{0}} \\
      &\cup \SetOf{\smallpmatrix{2 & 1 \\1 & 1} + \smallpmatrix{\lambda_1 & 0 \\\lambda_2 & 0}}
      {\smallcolvectwo{\lambda_1}{\lambda_2} \geq \smallcolvectwo{0}{0}} \\
      &\cup \SetOf{\smallpmatrix{2 & 1 \\1 & 1} + \smallpmatrix{-\lambda_1 & -\lambda_1 \\-\lambda_2 & -\lambda_2}}
      {\smallcolvectwo{\lambda_1}{\lambda_2} \geq \smallcolvectwo{0}{0}}
    \end{aligned}
  \]
  Recall that \enquote{$\leq$} and \enquote{$\geq$} refers to the lexicographic ordering.
  Due to this ordering, $\tropvariety_2(f)$ is not closed in the Euclidean topology.
  For example, consider the sequence of points 
  \[
  \begin{pmatrix}2 & 1 + c_k \\ 1 & 0 \end{pmatrix} \ \in \ (\TT_2)^2 \enspace ,
  \]
  where $c_k \rightarrow 0$ is a null sequence of positive reals.
  Each of these points are contained in $\tropvariety_2(f)$ but its limit is not.
\end{example}

Example \ref{ex:rank+two+hypersurface} highlights that rank two tropical hypersurfaces are not closed in the Euclidean topology.
Thus they do not have the structure of a polyhedral complex as rank one tropical hypersurfaces do.
Instead, we can consider polyhedral-like structures with respect to the lex-order topology on~$\TT_2$.

We recall the following notions from \cite{FosterRanganathan:2016degeneration,FosterRanganathan:2016hahn}.
There is a natural pairing \eqref{eq:pairing} which arises from considering the abelian group $\TT_2$ as a $\ZZ$-module.
A \emph{lex-halfspace} in ${(\TT_2)}^d$ is a set of the form
\[
\lex{H}_{s,q} \ = \ \SetOf{p \in {(\TT_2)}^d}{\langle s, p \rangle \leq q}
\]
for some fixed \emph{slope} $s \in \ZZ^d$ and \emph{affine constraint} $q \in \RR^2$.
Its \emph{boundary} is
\begin{equation}\label{eq:boundary}
  \SetOf{p \in {(\TT_2)}^d}{\langle s, p \rangle = q} \ = \ \lex{H}_{s,q} \cap \lex{H}_{-s,q} \enspace .
\end{equation}
Note that the slopes are integral vectors as we are considering Laurent polynomials (whose exponents lie in $\ZZ^d$) with coefficients in $\LL$, which is equipped with a rank two valuation that is \emph{not} discrete.
Thus $\ZZ^d$ arises as a factor of the domain of the pairing map \eqref{eq:pairing}.

\begin{definition}\label{def:lex-polyhedron}
A \emph{lex-polyhedron} $\lex{P}$ in ${(\TT_2)}^d$ is any intersection of finitely many lex-halfspaces
\begin{equation}\label{eq:lex-polyhedron}
   \lex{P} \ = \ \lex{H}_{s_1,q_1} \cap \cdots \cap \lex{H}_{s_r,q_r} \enspace .
\end{equation}
A \emph{face} of $\lex{P}$ is the intersection with any number of boundaries of the lex-halfspaces defining $\lex{P}$.
Its \emph{relative interior} $\inte(\lex{P})$ is the set of points contained in $\lex{P}$ but in no face of $\lex{P}$.
A \emph{lex-polyhedral complex} in ${(\TT_2)}^d$ is a finite collection $\{\lex{P}_j\}_{j \in J}$ of lex-polyhedra in ${(\TT_2)}^d$ such that every face of $\lex{P}_j$ also lies in the collection and the intersection of any two lex-polyhedra also lies in the collection.
\end{definition}

Note that \cite{FosterRanganathan:2016degeneration,FosterRanganathan:2016hahn} simply refer to these as \enquote{polyhedra}.
As we are also working with ordinary and tropical polyhedra, we use the prefix \enquote{lex} to stress the underlying lexicographical ordering, and use a bold typeface to differentiate it.
By \eqref{eq:boundary}, boundaries of lex-halfspaces and thus faces are lex-polyhedra.
Lex-polyhedra are necessarily closed in the order topology.

Given some subset $S \subseteq \supp(f)$, we define the \emph{support cell}
\begin{equation}
  \lex{P}_S(f) \ = \ \SetOf{p \in (\TT_2)^d}{S \subseteq \cD_p(f)}, \quad \text{for }S \subseteq \supp(f) \enspace .
\end{equation}
By definition, $\lex{P}_S = \lex{P}_S(f)$ is cut out by lex-halfspaces defined by the inequalities of the form
\begin{equation} \label{eq:hypersurface+lex+cell}
  \val_2 (\gamma_s) + \langle s , p\rangle \ \leq \ \val_2 (\gamma_s') + \langle s' , p\rangle \,, \quad \text{for }s \in S \text{ and } s' \in \supp(f)
\end{equation}
and so has the structure of a lex-polyhedron.

Note that for a non-generic polynomial $f$, there may exist $S$ such that $\trop_2(f)$ does not attain its minimum at precisely $S$ when evaluated at any point in $\lex{P}_S$.
Equivalently, there may exist $S, T$ such that $S \neq T$ but their support cells are equal as sets, i.e., $\lex{P}_S = \lex{P}_T$.
Any point in the support cells satisfies $S, T \subseteq \cD_p(f)$ and so they are equal to $\lex{P}_S \cup \lex{P}_T$ as a set.
This implies any support cell can be labelled by a unique maximal set, which we call the \emph{support set} i.e., $S$ is a support set of $f$ if $\lex{P}_S(f) = \lex{P}_T(f)$ implies $T \subseteq S$.
Note that the rank one analogue of support cells in $\TT^d$ are ordinary polyhedra; see \cite[Proposition 3.1.6]{Tropical+Book} and Question~\ref{qst:regular+refinement} below.
Support cells have some nice combinatorial properties:

\begin{lemma} \label{lem:support+cells}
  Let $S,T$ be support sets.
  \begin{enumerate}
  \item $\lex{P}_S \cap \lex{P}_T = \lex{P}_{S\cup T}$.
  \item $S \subset T$ if and only if $\lex{P}_T$ is a face of $\lex{P}_S$.
  \end{enumerate}
\end{lemma}

\begin{proof}
Denote inequalities of the form \eqref{eq:hypersurface+lex+cell} by $\alpha_{s,s'}$.
Consider the intersection $\lex{P}_S \cap \lex{P}_T$, it is cut out by the union of inequalities defining $\lex{P}_S$ and $\lex{P}_T$.
These are precisely the inequalities $\alpha_{s,s'}$ for $s \in S \cup T$, and is therefore equal to $\lex{P}_{S \cup T}$.
Furthermore, as $S, T$ are support sets, their union also is.

Any face of $\lex{P}_S$ is defined by setting certain inequalities of \eqref{eq:hypersurface+lex+cell} to equalities, or equivalently by adding the inequality $\alpha_{s',s}$.
If $T \supset S$ is the set of elements of $\supp(f)$ contained in an equality, then $\alpha_{s,s'}$ holds for all $s \in T$ and $s' \in \supp(f)$.
Therefore $T$ is a support set and $\lex{P}_T$ is the corresponding face of $\lex{P}_S$.
\end{proof}

\begin{remark} \label{rem:support+cells}
Lemma \ref{lem:support+cells} has two important consequences.
The first is that by associating support cells with their unique support set, each support cell has a canonical halfspace description via \eqref{eq:hypersurface+lex+cell}.
Furthermore, as faces of support cells are themselves support cells, this extends to a canonical inequality description of each face.
The second consequence is that as the faces of $\lex{P}_S$ are the points $p$ such that $S \subsetneq \cD_p(f)$, the relative interior of $\lex{P}_S$ is the set
\[
\inte(\lex{P}_S) \ = \ \SetOf{p \in (\TT_2)^d}{S = \cD_p(f)} \enspace .
\]
Note that this is not true if $S$ is not a support set.
\end{remark}

\begin{remark}
In topology the term ``cell'' is typically used for subsets of $\RR^{2\times d}$ which are homeomorphic with some closed Euclidean ball.
Here we deviate slightly based on the topology that we are using.
When working with $\RR^{2\times d}$ and the Euclidean topology, our cells will be convex polyhedra, whereas when working with $(\TT_2)^d$ and the order topology, our cells will be lex-polyhedra.
Note that in both cases, cells may be unbounded.
\end{remark}

\cite[Theorem 2.5.2]{FosterRanganathan:2016hahn} and \cite[Proposition 1.2]{NisseSottile:2013} show $\tropvariety_2(f)$ carries the structure of a lex-polyhedral complex. 
The following shows that this lex-polyhedral complex is labelled by subsets of monomials of $f$.

\begin{proposition} \label{prop:lex+polyhedral+hypersurface}
The rank two tropical hypersurface $\tropvariety_2(f)$ is a lex-polyhedral complex whose cells are of the form $\lex{P}_S$, where $S$ is a support set of cardinality greater than one.
\end{proposition}

\begin{proof}
Define the collection of lex-polyhedra 
\[
\lex{\Sigma} \ = \ \SetOf{\lex{P}_S}{S \text{ support set },|S|>1} \enspace .
\] 
By definition $\lex{\Sigma}$ and $\tropvariety_2(f)$ are equal as sets; it remains to show $\lex{\Sigma}$ is a lex-polyhedral complex.
By Lemma \ref{lem:support+cells}, $\lex{\Sigma}$ is closed under taking intersections and restricting to faces, therefore it is a lex-polyhedral complex.
\end{proof}

\begin{example} \label{exmp:lex+polyhedral+hypersurface}
We return to the polynomial $f = x_1 + tx_2 + t^2u$ from Example \ref{ex:rank+two+hypersurface}.
Its support is $\supp(f) = \{(0,0),(1,0),(0,1)\}$, and so $\tropvariety_2(f)$ is a lex-polyhedral complex in $(\TT_2)^2$ with three maximal lex-polyhedral cells:
\begin{align*}
\lex{P}_{\{(0,0),(1,0)\}} \ &= \
    \SetOf{\smallpmatrix{2 & 1 \\1 & 1} + \smallpmatrix{0 & \lambda_1 \\0 & \lambda_2}}
    {\smallcolvectwo{\lambda_1}{\lambda_2} \geq \smallcolvectwo{0}{0}} \\
\lex{P}_{\{(0,0),(0,1)\}} \ &= \ 
    \SetOf{\smallpmatrix{2 & 1 \\1 & 1} + \smallpmatrix{\lambda_1 & 0 \\\lambda_2 & 0}}
    {\smallcolvectwo{\lambda_1}{\lambda_2} \geq \smallcolvectwo{0}{0}} \\    
\lex{P}_{\{(1,0),(0,1)\}} \ &= \ 
    \SetOf{\smallpmatrix{2 & 1 \\1 & 1} + \smallpmatrix{-\lambda_1 & -\lambda_1 \\-\lambda_2 & -\lambda_2}}
    {\smallcolvectwo{\lambda_1}{\lambda_2} \geq \smallcolvectwo{0}{0}} \enspace .
\end{align*}
Their intersection is the common face $\lex{P}_{\{(0,0),(1,0),(0,1)\}} = \smallpmatrix{2 & 1 \\ 1 & 1}$.
\end{example}

\paragraph*{Convergent complex Hahn series.}

While Proposition \ref{prop:lex+polyhedral+hypersurface} gives a concrete description of rank two tropical hypersurfaces, the structure of lex-polyhedra is not as well understood as ordinary polyhedra. 
Here we approach these objects through convergent Hahn series.
So we consider a pair $\hahnconvergent{\RR}{t}$, $\hahnconvergent{\RR}{t,u}$ of convergent subrings of the field of Hahn series $\hahnseries{\RR}{t,u}$ in two indeterminates, $t$ and $u$, such that \eqref{eq:frac+diagram} commutes.
Writing $\hahnconvergent{\CC}{t,u} = \hahnconvergent{\RR}{t,u} + i\hahnconvergent{\RR}{t,u}$, that diagram naturally extends to the following commutative diagram of Laurent polynomial (semi-)rings.
\begin{equation}\label{eq:poly+diagram}
  \begin{tikzcd}
    \hahnconvergent{\CC}{t}[{\bm x}^{\pm}] \arrow{d}{\trop} \arrow{r}{\iota} &  \hahnconvergent{\CC}{t,u}[{\bm x}^{\pm}] \arrow{d}{\trop_2} \arrow[dashed]{r}{\pi_{u\mapsto\sigma}} & \hahnconvergent{\CC}{t}[{\bm x}^{\pm}] \arrow{d}{\trop} \\
    \TT[{\bm x}^{\pm}] \arrow{r}{\iota_*} & \TT_2[{\bm x}^{\pm}] \arrow{r}{\pi_{u*}} & \TT[{\bm x}^{\pm}]
  \end{tikzcd}
\end{equation}
Here ${\bm x}^{\pm}$ is shorthand for $x_1^\pm,\dots,x_d^\pm$.
Furthermore, $\iota, \iota_*, \pi_{u\mapsto\sigma}, \pi_{u*}$ are the same as in \eqref{eq:frac+diagram}, applied coefficientwise.
Again we also use $\pi_t$ and $\pi_{t*}$ despite the fact that the roles of $t$ and $u$ are not interchangeable in \eqref{eq:frac+diagram}; cf.\ Example~\ref{exmp:not-interchangeable}.

We say that $\sigma$ is \emph{admissible} for a polynomial $f$ if it is admissible for each of its coefficients.
In particular, we require $\sigma$ to be admissible to guarantee diagram~\eqref{eq:poly+diagram} commutes.
Note that here evaluating a series in $\hahnconvergent{\CC}{t,u}$, within its polyradius of convergence, is only defined for admissible positive \emph{real} values, despite that the coefficients are allowed to be complex numbers.
This yields a real-analytic function, which may not be holomorphic; however, see \cite{MullerStrohmaier:2014}.

\begin{example} \label{exmp:partial+eval}
  Consider the rank two bivariate polynomial $f = x_1 + tx_2 + t^2u$ in $\hahnfrac{\CC}{t,u}[x_1,x_2]$ from Example \ref{ex:rank+two+hypersurface}, where $\hahnfrac{\CC}{t,u}$ are complex Hahn fractions.
  The coefficients of $f$ converge to nonzero values for any positive evaluation.
  For instance, this gives the rank one polynomials
  \[
    \begin{array}{rcll}
      \pi_{u\mapsto 1}(f) &\!=\!& x_1 + tx_2 + t^2 & \in \hahnfrac{\CC}{t}[x_1,x_2] \quad \text{and}\\
      \pi_{t\mapsto 1}(f) &\!=\!&  x_1 + x_2 + u  & \in \hahnfrac{\CC}{u}[x_1,x_2] \enspace ,
    \end{array}
  \]
  obtained from evaluating at $u=1$ and $t=1$.
  Their rank one tropical hypersurfaces both are tropical lines in $\RR^2$.
\end{example}

\begin{example}\label{exmp:referee}
  We thank an anonymous referee for the following example.
  Consider the rank two univariate polynomial $f = t(u-u^2) + t^2 - x$ in $\hahnfrac{\CC}{t,u}[x]$.
  We see that
  \[
    \begin{array}{rcll}
      \pi_{u\mapsto\sigma}(f) &\!=\!& t(\sigma-\sigma^2) + t^2 - x & \in \hahnfrac{\CC}{t}[x] \enspace ,
    \end{array}
  \]
  is admissible for all positive $\sigma \neq 1$.
  For $\sigma = 1$, the leading constant term is killed, and so diagram~\eqref{eq:poly+diagram} does not commute for this value of $\sigma$:
  \[
    \begin{array}{rclcl}
      \trop(\pi_{u\mapsto 1}(f)) &\!=\!& \trop(t^2 - x) &\!=\!& 2 + x  \\
      \pi_{u*}(\trop_2(f)) &\!=\!& \pi_{u*}((1,1) + x)  &\!=\!& 1 + x \enspace .
    \end{array}
  \]
\end{example}

For clarity, we use $\tropvariety$ rather than $\tropvariety_2$ to denote tropical hypersurfaces where the underlying field has rank one valuation.
As $\pi_{u\mapsto\sigma}(f)$ and $\pi_{t\mapsto\rho}(f)$ are polynomials over an algebraically closed field with a rank one valuation, their tropical hypersurfaces $\tropvariety(\pi_{u\mapsto\sigma}(f))$ and $\tropvariety(\pi_{t\mapsto\rho}(f))$ are ordinary polyhedral complexes.
However, the underlying fields are different and so these tropical hypersurfaces sit in different ambient spaces that we denote by $\RR_t^d$ and $\RR_u^d$ respectively.
Using Theorem \ref{thm:rank2+kapranov} and the commutative diagram \eqref{eq:poly+diagram}, we may view the entire space
\[
  \RR^{2\times d} \ = \ \pi_{u*}(\RR^{2\times d}) \times \pi_{t*}(\RR^{2\times d}) \ = \ \RR_t^d \times \RR_u^d
\]
as their Cartesian product.

As noted previously, $\tropvariety_2(f)$ is not closed in the Euclidean topology and so is not a polyhedral complex.
However, we can still use the additional structure of $\tropvariety(\pi_{u\mapsto\sigma}(f))$ and $\tropvariety(\pi_{t\mapsto\rho}(f))$ to describe $\tropvariety_2(f)$.

The \emph{(relative) interior} of an ordinary polyhedron $P$ is the set of points $\inte(P)$ contained in $P$ but no proper face of $P$.
Equivalently, it is the set cut out by the defining equalities and inequalities of $P$, where any proper inequalities are changed to strict inequalities.
By removing its boundary, the interior of a polyhedron is not closed in the Euclidean topology, and so this is what we shall use to describe $\tropvariety_2(f)$.
Note that the interior of a polyhedron is open if and only if it is full dimensional.

Let $f=\sum\gamma_s x^s$.
For $T \subseteq \supp(f)$, we denote the restriction of $f$ to the monomials labelled by $T$ by $f_T = \sum_{s \in T} \gamma_s x^s$.
We denote the support cells of $f_T$ with support set $S$ as $P_{S,T}$, where the extra index emphasises the restriction on the support of $f$.
The following is our first main result.
The Example~\ref{exmp:referee} shows that the admissibility assumption is crucial.

\begin{theorem}\label{thm:union+polyhedra}
  Let $f \in \hahnconvergent{\CC}{t,u}[x_1^{\pm},\dots,x_d^{\pm}]$ be a $d$-variate Laurent polynomial with admissible partial evaluations $t\mapsto\rho$ and $u\mapsto\sigma$.
  The rank two tropical hypersurface $\tropvariety_2(f)$ is the finite disjoint union
  \[
    \tropvariety_2(f) \ = \ \bigsqcup_{S} \bigsqcup_{T \supseteq S} \bigl(\inte(Q_T) \times \inte(R_{S,T}) \bigr)
  \]
of interiors of ordinary polyhedra in $\RR^{2\times d}$, where $Q_T$ and $R_{S,T}$ are support cells of the rank one tropical hypersurfaces $\tropvariety(\pi_{u\mapsto\sigma}(f))$ in $\RR_t^d$ and $\tropvariety(\pi_{t\mapsto\rho}(f_T))$ in $\RR_u^d$, respectively.
\end{theorem}
\begin{proof}
By Proposition \ref{prop:lex+polyhedral+hypersurface}, $\tropvariety_2(f)$ is a lex-polyhedral complex of support cells $\lex{P}_S$ as $S$ runs over all support sets of $f$ of cardinality greater than one.
In particular, this becomes a disjoint union if we restrict to the relative interiors of $\lex{P}_S$; by Remark \ref{rem:support+cells} these are the points $p$ such that $\trop_2(f)(p)$ attains its minimum at precisely the monomials labelled by $S$.
We claim that $\inte(\lex{P}_S) = \bigsqcup_{T \supseteq S} \left(\inte(Q_T) \times \inte(R_{S,T})\right)$.

The point $p$ is contained in $\inte(\lex{P}_S)$ if and only if $\trop_2(f)(p)$ attains its minimum at precisely the monomials labelled by $S$ i.e.,
\begin{equation}\label{eq:S+tight}
\val_2(\gamma_s) + \langle s , p \rangle \ \leq \ \val_2(\gamma_{s'}) + \langle s' , p \rangle \ , \ \text{for all } s \in S \text{ and } s' \in \supp(f)
\end{equation}
with equality if and only if $s' \in S$.
Taking into consideration the lexicographical ordering on $\TT_2$, we can consider its coordinates separately to derive conditions on $\pi_{t*}(p)$ and $\pi_{u*}(p)$.

Consider condition~\eqref{eq:S+tight} restricted to the first coordinate.
Due to the lexicographical ordering on $\TT_2$, equality is attained in the first coordinate for some superset $T \supseteq S$, where
\[
  \begin{aligned}
    T \ &= \ \argmin_{s \in \supp(f)}\left( \pi_{u*}(\val_2(\gamma_s)) + \pi_{u*}(\langle s , p \rangle) \right) \\
    &= \ \argmin_{s \in \supp(f)}\left(\val(\pi_{u\mapsto\sigma}(\gamma_s)) + \sum_{i=1}^d s_i p_{1i}\right) \enspace .
  \end{aligned}
\]
This labels the precise set of monomials that $\trop(\pi_{u\mapsto\sigma}(f))(\pi_{u*}(p))$ attains its minimum at.
Therefore we can deduce that $\pi_{u*}(p)$ is contained in the interior of the support cell $Q_T$ of $\tropvariety(\pi_{u\mapsto\sigma}(f))$.

For condition~\eqref{eq:S+tight} to hold, the restriction of~\eqref{eq:S+tight} to the second coordinate to be a strict inequality for all $s \in S$ and $s' \in T \setminus S$, and an equality for all $s,s' \in S$.
This is equivalent to
\[
  \begin{aligned}
    S \ &= \ \argmin_{s \in T}\left( \pi_{t*}(\val_2(\gamma_s)) + \pi_{t*}(\langle s , p \rangle) \right) \\
    &= \ \argmin_{s \in T}\left(\val(\pi_{t\mapsto\rho}(\gamma_s)) + \sum_{i=1}^d s_i p_{2i}\right) \enspace .
\end{aligned}
\]
This labels the precise set of monomials that $\trop(\pi_{t\mapsto\rho}(f_T))(\pi_{t*}(p))$ attains its minimum at.
Therefore we can deduce that $\pi_{t*}(p)$ is contained in the interior of the support cell $R_{S,T}$ of $\tropvariety(\pi_{t\mapsto\rho}(f_T))$.

It remains to show each part of the disjoint union is the interior of a polyhedron, or explicitly that $\inte(Q_T \times R_{S,T}) = \inte(Q_T) \times \inte(R_{S,T})$.
As $Q_T$ and $R_{S,T}$ are in orthogonal ambient spaces, the union of their defining equalities and inequalities cut out $Q_T \times R_{S,T}$.
Changing the inequalities to strict inequalities gives the desired result.
\end{proof}

Since the order topology is finer than the Euclidean topology, the Euclidean closure becomes larger.
\begin{corollary} \label{cor:hypersurface+closure}
  With the notation of Theorem~\ref{thm:union+polyhedra}: the closure of $\tropvariety_2(f)$ in the Euclidean topology is the finite union
  \[
    \overline{\tropvariety_2(f)} \ = \ \bigcup_S \bigcup_{T \supseteq S} \bigl(Q_T \times R_{S,T} \bigr)
  \]
  of polyhedra in $\RR^{2\times d}$.
\end{corollary}
\begin{proof}
  As $Q_T \times R_{S,T} = \overline{\inte(Q_T) \times \inte(R_{S,T})}$, the result follows from Theorem \ref{thm:union+polyhedra} using the fact that the closure of a finite union of sets equals the union of their closures.
\end{proof}

\begin{remark} \label{rem:hypersurface+characterisation}
Building on Theorem \ref{thm:union+polyhedra} and Corollary \ref{cor:hypersurface+closure}, one can give a slightly different characterisation of $\tropvariety_2(f)$ and its closure.
Letting $T$ range over support sets of $\pi_{u\mapsto\sigma}(f)$ and $S$ over support sets of $\pi_{t\mapsto\rho}(f_T)$, we get
\begin{align*}
\tropvariety_2(f) \ &= \ \bigsqcup_{S} \bigsqcup_{T \supseteq S} \bigl(\inte(Q_T) \times \inte(R_{S,T}) \bigr) \\
					&= \ \bigsqcup_{T} \bigl(\inte(Q_T) \times \bigsqcup_{S \subseteq T} \inte(R_{S,T})\bigr) \\
					&= \ \bigsqcup_{T} \bigl(\inte(Q_T) \times \tropvariety(\pi_{t\mapsto\rho}(f_T)) \bigr) \enspace .
\end{align*}
Taking the closure in the Euclidean topology gives the expression $\overline{\tropvariety_2(f)} = \bigcup_{T} \bigl(Q_T \times \tropvariety(\pi_{t\mapsto\rho}(f_T)) \bigr)$.
These alternative characterisations will be of use for Section \ref{sec:stable}.
\end{remark}

To close this section, we give two examples to demonstrate that rank two tropical hypersurfaces are quite different from their rank one counterparts, even when taking their closure in the Euclidean topology.
Example \ref{ex:cell+complex} demonstrates the closure of a rank two tropical hypersurface is not a polyhedral complex, as polyhedra may not intersect at their faces.
Example \ref{ex:differing+max+dim} shows the closure of a rank two tropical hypersurface does not satisfy a purity condition, as the polyhedra that are maximal with respect to inclusion may not be of the same dimension.

\begin{example} \label{ex:cell+complex}
We return to the rank two tropical hypersurface of the polynomial $f = x_1 + tx_2 + t^2u$ from Examples \ref{ex:rank+two+hypersurface}, \ref{exmp:lex+polyhedral+hypersurface} and \ref{exmp:partial+eval}.
As its coefficients are monomials in $t$ and $u$, the partial evaluations of $f$ are defined at the admissible values $\rho = \sigma = 1$.
Let $T = \{(0,0),(0,1)\}$, and consider the support cell
\[
Q_T \ = \ \SetOf{(2 + \lambda_1,1)}{\lambda_1 \geq 0}
\]
of the tropical line $\tropvariety(\pi_{u\mapsto 1}(f))$ in $\RR_t^2$.
The polynomial $\pi_{t\mapsto 1}(f_T) = x_2 + u$ defines a rank 1 tropical hypersurface with a single support cell
\[
R_{S,T} \ = \ \SetOf{(\lambda_2,1)}{\lambda_2 \in \RR} \enspace ,
\]
in $\RR_u^2$, where $S = \{(0,0),(0,1)\}$.
By Corollary \ref{cor:hypersurface+closure}, the product of these two polyhedra
\[
Q_T \times R_{S,T} \ = \ \SetOf{\smallpmatrix{2 & 1 \\1 & 1} + \smallpmatrix{\lambda_1 & 0 \\\lambda_2 & 0}}{\lambda_1 \geq 0, \lambda_2 \in \RR} \subset \RR^{2\times 2}
\]
is a polyhedron in $\overline{\tropvariety_2(f)}$.
Ranging over all support sets $S$ and $T$, the closure of $\tropvariety_2(f)$ in the Euclidean topology is the union
\[
  \begin{split}
    \overline{\tropvariety_2(f)}
    \ = \ &\SetOf{\smallpmatrix{2 & 1 \\1 & 1} + \smallpmatrix{\lambda_1 & 0 \\\lambda_2 & 0}}{\lambda_1 \geq 0, \lambda_2 \in \RR} \, \cup \, \SetOf{\smallpmatrix{2 & 1 \\1 & 1} + \smallpmatrix{0 & \lambda_1 \\0 & \lambda_2}}{\lambda_1 \geq 0, \lambda_2 \in \RR} \\
    & \cup \, \SetOf{\smallpmatrix{2 & 1 \\1 & 1} + \smallpmatrix{\lambda_1 & \lambda_1 \\\lambda_2 & \lambda_2}}{\lambda_1 \leq 0, \lambda_2 \in \RR}
  \end{split}
\]
of three ordinary halfplanes in $\RR^{2\times 2}$.
Note that this is not an ordinary polyhedral complex as the polyhedra do not intersect at faces.
The joint intersection of the three ordinary halfplanes is the point $\smallpmatrix{2 & 1 \\ 1 & 1}$, but this is not a (zero-dimensional) face of any of them.
\end{example}

\begin{example} \label{ex:differing+max+dim}
  Consider the polynomial $f = ux_1x_2 + x_1 + x_2 + 1$, whose vanishing locus is a conic.
  The closure of its rank two tropical hypersurface is the union of ordinary polyhedra:
  \[
    \begin{split}
      \overline{\tropvariety_2(f)}  \ = \ & \SetOf{\smallpmatrix{ \lambda_1 & 0 \\ \lambda_2 & 0 }}{\lambda_1 \geq 0, \lambda_2 \in \RR} \, \cup \, \SetOf{\smallpmatrix{ 0 & \lambda_1 \\ 0 & \lambda_2 }}{\lambda_1 \geq 0, \lambda_2 \in \RR} \\
      & \cup \, \SetOf{\smallpmatrix{ 0 & 0 \\ \lambda & \lambda }}{\lambda \in [-1,0]} \, \cup \, \SetOf{\smallpmatrix{ \lambda_1 & 0 \\ \lambda_2 & -1 }}{\lambda_1 \leq 0, \lambda_2 \in \RR} \\
      & \cup \, \SetOf{\smallpmatrix{ 0 & \lambda_1 \\ -1 & \lambda_2 }}{\lambda_1 \leq 0, \lambda_2 \in \RR} \enspace .
    \end{split}
  \]
  We say a finite union of polyhedra is \emph{pure} if all its maximal polyhedra (with respect to inclusion) have the same dimension.
  This generalises a notion commonly used for polyhedral complexes; in fact, it is the same if applied to the polyhedral complex obtained by taking the common refinement of the finitely many given polyhedra.
  Observe that $\overline{\tropvariety_2(f)}$ is not pure, as the maximal polyhedra are all two-dimensional, except for the line segment 
  \[
  \SetOf{\smallpmatrix{ 0 & 0 \\ \lambda & \lambda }}{\lambda \in [-1,0]} \enspace .
  \]
  This can be decomposed as the product of support cells
  \[
  Q_T \times R_{S,T} \ = \ \{(0,0)\} \times \SetOf{(\lambda,\lambda)}{\lambda \in \RR}
  \]
  where $T = \{(0,0),(1,0),(0,1),(1,1)\}$ and $S = \{(1,0),(0,1)\}$.
  In particular, $S \subset T$ implies $\dim(Q_T) < \dim(R_{S,T})$.
  However, the pairs of support cells in the decomposition of the other maximal polyhedra have equal support sets, and therefore the same dimension.
\end{example}

\paragraph*{Non-surjective valuations.}

Throughout we have insisted the valuation map $\val_2:\LL \rightarrow \TT_2$ is a surjective valuation.
For rank one valuations, such assumptions are not required, furthermore there is existing work that does not rely on these assumptions for higher rank valuations.
We close this section by comparing our approach to existing literature, and discussing the issues that can arise when not using surjective higher rank valuations.

Foster and Ranganathan \cite{FosterRanganathan:2016hahn} and Banerjee \cite{Banerjee:2015} both study notions of higher rank tropical geometry; in both cases the group of values is $\TT_m$ (or a discrete subgroup).
Banerjee considers the tropicalisation of subvarieties of the torus over $m$-dimensional local fields with discrete valuation, while Foster and Ranganathan consider a generalisation of Berkovich analytification.
As we shall see, Banerjee's tropicalisation is via valuations that do not surject onto $\TT_m$ and is therefore not comparable to ours.
However, both are special cases of the tropicalisation in \cite{FosterRanganathan:2016hahn}.
In particular, for $m=2$ our $\tropvariety_2(f)$ from Definition~\ref{def:hypersurface} is covered in \cite{FosterRanganathan:2016hahn}.

There is a conceptual difference between the approach of Foster and Ranganathan and Banerjee's approach.
Banerjee begins with small fields and discrete valuations and then takes algebraic and topological closures to ``fill in gaps'', while Foster and Ranganathan begin with larger fields, via Hahn analytification, to avoid taking topological closures.
Our approach is in the same spirit as Foster and Ranganathan's.
While either approach behaves well for $m=1$, the following shows that topological closure operations go awry when $m>1$ and thus need to be dealt with carefully.

To see this, first let us very briefly describe the setup of \cite{Banerjee:2015}.
Any \emph{$m$-dimensional local field} $\KK$, in the sense of \cite[Definition 3.1]{Banerjee:2015}, admits a valuation $\nu^{\KK} : \KK^{\times} \rightarrow \Gamma^{\KK}$ where $\Gamma^{\KK} \cong \ZZ^m$ with the lexicographical ordering.
For any finite field extension $\LL$ of~$\KK$, this valuation extends to a valuation $\nu^{\LL} : \LL^{\times} \rightarrow \Gamma^{\LL}$.
This allows us to extend $\nu^{\KK}$ to the algebraic closure of $\KK$, becoming the surjective map $\nu :(\KK^{\rm al})^{\times} \rightarrow \Gamma_{\QQ} \cong \QQ^m$ where $\Gamma_{\QQ}$ is the direct limit of all groups $\Gamma^{\LL}$ taken over all finite field extensions $\LL$ of $\KK$.
Finally, we let $\Gamma_{\RR} := \Gamma_{\QQ} \otimes_{\QQ} \RR \cong \RR^m$ and extend the codomain of $\nu$ to $\Gamma_{\RR}$.
One then considers subvarieties of the $d$-dimensional algebraic torus over $\KK$ and their images in $\nu$.

Banerjee's notion of a tropical hypersurface is the same as Aroca's \cite{Aroca:2010b}, and this agrees with Definition~\ref{def:hypersurface}.
Now \cite[Theorem 5.3]{Banerjee:2015} claims that $\tropvariety_m(f)$ is equal to
\[ \overline{\SetOf{\nu(p)}{p \in \mathcal{X}_f}} \enspace ,\]
where $\mathcal{X}_f$ is the hypersurface in the algebraic torus defined by $f$.
Unfortunately, in which topology the closure is taken in is not specified.
The discussion in \cite[Section 2.3]{FosterRanganathan:2016hahn} erroneously assumes it is the Euclidean topology.
However, the resulting set contains $\tropvariety_m(f)$ but is too large and contains points where $\trop_m(f)$ is linear.
Note that Banerjee's definition of a polyhedron \cite[Notation 4.1.(v)]{Banerjee:2015} generalises our definition of a lex-polyhedron slightly by replacing $\ZZ^m$ by any totally order group $\Gamma$.
Furthermore, \cite[Example 5.11]{Banerjee:2015} is a computation of a rank two tropical hypersurface, similar to our Example~\ref{ex:rank+two+hypersurface}, and is not closed in the Euclidean topology.

However, it is worth noting that taking the order topology does not fix the claim made in \cite[Theorem 5.3]{Banerjee:2015}.
The image of the valuation $\nu$ is isomorphic to $\QQ^m$ with the lexicographical ordering.
In the order topology, $\QQ^m$ is not dense in $\RR^m$, as its closure does not contain any elements of the form $(a_1,\dots,a_m)$ where $a_1$ is irrational.
Therefore the closure in the order topology is contained in $\tropvariety_m(f)$ but is too small.

\section{Stable intersection}
\label{sec:stable}
\noindent
In this section, we use the higher rank machinery developed so far to obtain a new description of the stable intersection of rank one tropical hypersurfaces.
To do so, we must first consider the structure of rank two tropical hypersurfaces determined by polynomials with coefficients in $\hahnconvergent{\CC}{t}$, a convergent subring of $\hahnseries{\CC}{t}$.

We recall the following polyhedral definition.
Fix some polyhedral complex $\Sigma$ and let $P$ be a cell in $\Sigma$.
The \emph{star} of $P$ is the fan spanned by the cells of $\Sigma$ containing $P$; more precisely,
\begin{equation} \label{eq:star}
  \sta(P) \ = \ \bigcup_{Q\in\Sigma, \, Q\supseteq P}\SetOf{\lambda(q-p)}{\lambda \geq 0, \, p \in P, \, q \in Q} \enspace .
\end{equation}

Let $f$ be a Laurent polynomial in $\hahnconvergent{\CC}{t}[x_1^\pm,\dots,x_d^\pm]$.
Under the embedding $\iota$, we can also consider $f$ as a polynomial in $\hahnconvergent{\CC}{t,u}[x_1^\pm,\dots,x_d^\pm]$ with an associated rank two tropical hypersurface.
We arrive at another consequence of Theorem~\ref{thm:union+polyhedra}.

\begin{corollary}\label{cor:KK+hypersurface+structure}
  Let $f \in \hahnconvergent{\CC}{t}[x_1^{\pm},\dots,x_d^{\pm}]$ be a $d$-variate Laurent polynomial.
  The rank two tropical hypersurface $\tropvariety_2(f)$ is the disjoint finite union
  \[
    \tropvariety_2(f) \ = \ \bigsqcup_S \bigl( \inte(P_S) \times\sta(P_S) \bigr)
  \]
 in $\RR^{2\times d}$, where $P_S$ is a support cell of $\tropvariety(f)$ in $\RR_t^d$ and $\sta(P_S)$ is embedded in $\RR_u^d$.
 \end{corollary}

\begin{proof}
  Clearly, this is a special case of Remark \ref{rem:hypersurface+characterisation} where $f$ agrees with $\pi_{u\mapsto\sigma}(f)$.
  We infer that $\tropvariety_2(f)$ is the disjoint union $\inte(P_S) \times \tropvariety(\pi_{t\mapsto\rho}(f_S))$.
  Since $\pi_{t\mapsto\rho}(f_S)$ has constant coefficients its tropical hypersurface is a fan.
  By \cite[Theorem 3.5.6]{Tropical+Book} this is the \emph{recession fan} of $\tropvariety(f_S)$, and in this case it agrees with $\sta(P_S)$.
\end{proof}

While $\tropvariety_2(f)$ is naturally endowed with the order topology, the power of the Euclidean topology is that $\overline{\tropvariety_2(f)}$ has a far cleaner structure.
This will be crucial for our main result of this section, Theorem~\ref{thm:stable+intersection}.

\begin{corollary}\label{cor:KK+hypersurface+closure}
The closure of $\tropvariety_2(f)$ in the Euclidean topology is the finite union
\[
\overline{\tropvariety_2(f)} \ = \ \bigcup_S \bigl( P_S \times L_S \bigr)
\]
of polyhedra in $\RR^{2\times d}$, where $P_S$ is a maximal support cell of $\tropvariety(\pi_{u\mapsto\sigma}(f))$ in $\RR_t^d$ and $L_S$ is the linear space equal to the affine span of $P_S$ translated to the origin in $\RR_u^d$.
\end{corollary}
\begin{proof}
Remark \ref{rem:hypersurface+characterisation} and Corollary \ref{cor:KK+hypersurface+structure} imply that $\overline{\tropvariety_2(f)}$ equals the union $\bigcup \bigl( P_S \times \sta(P_S) \bigr)$.
Each cell of $\sta(P_S)$ is labelled by some $T \subseteq S$ corresponding to $P_T \supseteq P_S$.
Note that if $P_S$ is a maximal support cell of $\tropvariety(\pi_{u\mapsto\sigma}(f))$, $\sta(P_S)$ is simply the linear space $L_S$.
Furthermore, if $P_S$ is not a maximal support cell of $\tropvariety(\pi_{u\mapsto\sigma}(f))$, then the maximal cell of $\sta(P_S)$ labelled by $T \subset S$ is contained in $L_T$.
Therefore we can restrict the union to just the maximal support cells, giving the desired result.
\end{proof}

\begin{example}\label{exmp:elliptic}
  Consider the degree three polynomial
  \[
    f \ = \ 1 + t(x+y) + t^3xy + t^5(x^2+y^2) + t^9(x^2y+xy^2) + t^{15}(x^3+y^3)
  \]
  in $\hahnfrac{\CC}{t}[x,y]$.
  It describes an elliptic curve, whose rank one tropicalisation is shown in Figure~\ref{fig:elliptic}.
  When we view $f$ as a polynomial with coefficients in $\hahnfrac{\CC}{t,u}$, Corollary~\ref{cor:KK+hypersurface+structure} describes the resulting rank two tropical curve.
  The partial evaluation $\pi_{u\mapsto\sigma}(f)$ equals $f$, and $\pi_{t\mapsto\rho}(f)$ has constant coefficients.
  For instance, let us look at the cell marked ``$P_S$'' in Figure~\ref{fig:elliptic} where $S =\{(0,1),(1,1)\}$, we get $f_S=ty+t^3xy$.
  It follows that $L_S=\tropvariety(f_S)$ is the $y$-axis, and this is also the only cell in that tropical hypersurface.
\end{example}

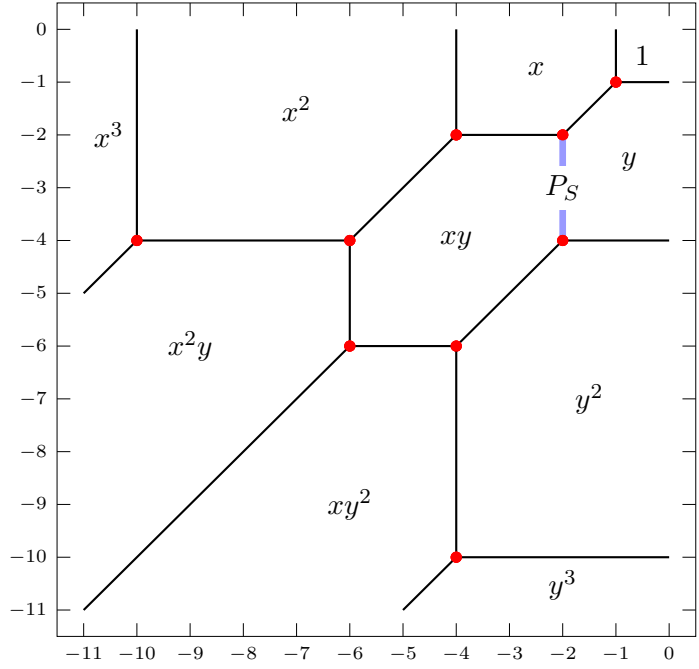
\begin{figure}
\begingroup

\newcommand\gridoffset{0.5}
\newcommand\ticklength{0.25}
\newcommand\gridfont{\tiny}

\begin{tikzpicture}[x  = {(1cm,0cm)},
                    y  = {(0cm,1cm)},
                    scale = 0.7,
                    color = {black}]

  \definecolor{pointcolor}{rgb}{ 1,0,0 }
  \tikzstyle{pointstyle} = [fill=pointcolor]
  \tikzstyle{linestyle} = [black, thick]
  \tikzstyle{marked} = [blue!40, line width=2.5pt]
  \tikzstyle{gridstyle} = [black, very thin]

  \draw[gridstyle] ($ (0,0) + (\gridoffset,\gridoffset) $) -- ($ (0,-11) + (\gridoffset,-\gridoffset) $)  -- ($ (-11,-11) - (\gridoffset,\gridoffset) $) -- ($ (-11,0) + (-\gridoffset,\gridoffset) $) -- cycle;

  \foreach \i in {-11,-10,...,0}{
    \coordinate (rightbd) at ($ (0,\i) + (\gridoffset,0) $);
    \draw[gridstyle] (rightbd) -- + (-\ticklength,0);

    \coordinate (leftbd) at ($ (-11,\i) - (\gridoffset,0) $);
    \draw[gridstyle] (leftbd) -- + (\ticklength,0);
    \node at (leftbd) [left] { \gridfont $\i$ };

    \coordinate (topbd) at ($ (\i,0) + (0,\gridoffset) $);
    \draw[gridstyle] (topbd) -- + (0,-\ticklength);

    \coordinate (bottombd) at ($ (\i,-11) - (0,\gridoffset) $);
    \draw[gridstyle] (bottombd) -- + (0,\ticklength);
    \node at (bottombd) [below] { \gridfont $\i$ };
  }
  
  \coordinate (v0_unnamed__1) at (-1, 0);
  \coordinate (v1_unnamed__1) at (-1, -1);

  \draw[linestyle] (v1_unnamed__1) -- (v0_unnamed__1);

  \fill[pointcolor] (v1_unnamed__1) circle (3 pt);

  \coordinate (v0_unnamed__2) at (-4, 0);
  \coordinate (v1_unnamed__2) at (-4, -2);

  \draw[linestyle] (v1_unnamed__2) -- (v0_unnamed__2);

  \fill[pointcolor] (v1_unnamed__2) circle (3 pt);

  \coordinate (v0_unnamed__3) at (-10, 0);
  \coordinate (v1_unnamed__3) at (-10, -4);

  \draw[linestyle] (v1_unnamed__3) -- (v0_unnamed__3);

  \fill[pointcolor] (v1_unnamed__3) circle (3 pt);

  \coordinate (v0_unnamed__4) at (0, -1);
  \coordinate (v1_unnamed__4) at (-1, -1);

  \draw[linestyle] (v1_unnamed__4) -- (v0_unnamed__4);

  \fill[pointcolor] (v1_unnamed__4) circle (3 pt);

  \coordinate (v0_unnamed__5) at (-1, -1);
  \coordinate (v1_unnamed__5) at (-2, -2);

  \draw[linestyle] (v1_unnamed__5) -- (v0_unnamed__5);

  \fill[pointcolor] (v1_unnamed__5) circle (3 pt);
  \fill[pointcolor] (v0_unnamed__5) circle (3 pt);

  \coordinate (v0_unnamed__6) at (-2, -2);
  \coordinate (v1_unnamed__6) at (-4, -2);

  \draw[linestyle] (v1_unnamed__6) -- (v0_unnamed__6);

  \fill[pointcolor] (v1_unnamed__6) circle (3 pt);
  \fill[pointcolor] (v0_unnamed__6) circle (3 pt);

  \coordinate (v0_unnamed__7) at (-4, -2);
  \coordinate (v1_unnamed__7) at (-6, -4);

  \draw[linestyle] (v1_unnamed__7) -- (v0_unnamed__7);

  \fill[pointcolor] (v1_unnamed__7) circle (3 pt);
  \fill[pointcolor] (v0_unnamed__7) circle (3 pt);

  \coordinate (v0_unnamed__8) at (-2, -2);
  \coordinate (v1_unnamed__8) at (-2, -4);

  \draw[marked] (v1_unnamed__8) -- (v0_unnamed__8);
  \node at (-2,-3) [fill=white] { $P_S$ };
  
  \fill[pointcolor] (v1_unnamed__8) circle (3 pt);
  \fill[pointcolor] (v0_unnamed__8) circle (3 pt);

  \coordinate (v0_unnamed__9) at (-6, -4);
  \coordinate (v1_unnamed__9) at (-10, -4);

  \draw[linestyle] (v1_unnamed__9) -- (v0_unnamed__9);

  \fill[pointcolor] (v1_unnamed__9) circle (3 pt);
  \fill[pointcolor] (v0_unnamed__9) circle (3 pt);
 
  \coordinate (v0_unnamed__10) at (-10, -4);
  \coordinate (v1_unnamed__10) at (-11, -5);

  \draw[linestyle] (v1_unnamed__10) -- (v0_unnamed__10);

  \fill[pointcolor] (v0_unnamed__10) circle (3 pt);

  \coordinate (v0_unnamed__11) at (-6, -4);
  \coordinate (v1_unnamed__11) at (-6, -6);

  \draw[linestyle] (v1_unnamed__11) -- (v0_unnamed__11);

  \fill[pointcolor] (v1_unnamed__11) circle (3 pt);
  \fill[pointcolor] (v0_unnamed__11) circle (3 pt);

  \coordinate (v0_unnamed__12) at (0, -4);
  \coordinate (v1_unnamed__12) at (-2, -4);

  \draw[linestyle] (v1_unnamed__12) -- (v0_unnamed__12);

  \fill[pointcolor] (v1_unnamed__12) circle (3 pt);

  \coordinate (v0_unnamed__13) at (-2, -4);
  \coordinate (v1_unnamed__13) at (-4, -6);

  \draw[linestyle] (v1_unnamed__13) -- (v0_unnamed__13);

  \fill[pointcolor] (v1_unnamed__13) circle (3 pt);
  \fill[pointcolor] (v0_unnamed__13) circle (3 pt);

  \coordinate (v0_unnamed__14) at (-4, -6);
  \coordinate (v1_unnamed__14) at (-6, -6);

  \draw[linestyle] (v1_unnamed__14) -- (v0_unnamed__14);

  \fill[pointcolor] (v1_unnamed__14) circle (3 pt);
  \fill[pointcolor] (v0_unnamed__14) circle (3 pt);

  \coordinate (v0_unnamed__15) at (-6, -6);
  \coordinate (v1_unnamed__15) at (-11, -11);

  \draw[linestyle] (v1_unnamed__15) -- (v0_unnamed__15);

  \fill[pointcolor] (v0_unnamed__15) circle (3 pt);

  \coordinate (v0_unnamed__16) at (-4, -6);
  \coordinate (v1_unnamed__16) at (-4, -10);

  \draw[linestyle] (v1_unnamed__16) -- (v0_unnamed__16);

  \fill[pointcolor] (v1_unnamed__16) circle (3 pt);
  \fill[pointcolor] (v0_unnamed__16) circle (3 pt);

  \coordinate (v0_unnamed__17) at (0, -10);
  \coordinate (v1_unnamed__17) at (-4, -10);

  \draw[linestyle] (v1_unnamed__17) -- (v0_unnamed__17);

  \fill[pointcolor] (v1_unnamed__17) circle (3 pt);

  \coordinate (v0_unnamed__18) at (-4, -10);
  \coordinate (v1_unnamed__18) at (-5, -11);

  \draw[linestyle] (v1_unnamed__18) -- (v0_unnamed__18);

  \fill[pointcolor] (v0_unnamed__18) circle (3 pt);

  \node at (-0.5,-0.5) { $1$ };
  \node at (-2.5,-0.75) { $x$ };
  \node at (-7,-1.5) { $x^2$ };
  \node at (-11,-2) [right] { $x^3$ };
  \node at (-0.75,-2.5) { $y$ };
  \node at (-1.5,-7) { $y^2$ };
  \node at (-2,-11) [above] { $y^3$ };
  \node at (-4,-4) { $xy$ };
  \node at (-9,-6) { $x^2y$ };
  \node at (-6,-9) { $xy^2$ };
\end{tikzpicture}
\endgroup

  \caption{Tropical elliptic curve with the one-dimensional cell $P_S$ marked; cf.\ Example~\ref{exmp:elliptic}.  Each region is labelled with its supporting monomial.}
  \label{fig:elliptic}
\end{figure}

To develop a new description of stable intersection, we introduce the following notion of perturbation on the level of convergent Hahn series.

\begin{definition}
  Let $\beta>0$ be a fixed transcendental number.
  The \emph{$u$-perturbation} of $f$ by $\beta$ is the polynomial $f^u \in \hahnconvergent{\CC}{t,u}[x_1^{\pm},\dots,x_d^{\pm}]$ obtained from $f$ by the $d$ linear substitutions $x_k \mapsto u^{\beta^k}x_k$.
\end{definition}
We are interested in the effect of the $u$-perturbation to the tropicalisation of $f$.
As $\val(u) < \val(t)$, the variable $u$ can be considered an infinitesimal perturbation to the coefficients of $f$.
Explicitly, the $u$-perturbation of the term $\gamma_s x^s$, which is a $d$-variate Laurent monomial whose single coefficient $\gamma_s$ lies in $\hahnconvergent{\CC}{t}$, equals
\[
  \gamma_s u^{s_1\beta+s_2\beta^2+\dots+s_d\beta^d} x^s \enspace .
\]
Its rank two tropicalisation is 
\[
  \bigl(\val(\gamma_s), \sum s_i\beta^i\bigr) + s_1x_1+\dots+s_dx_d  \enspace .
\]
Since $\beta$ is transcendental, the expression $\sum s_i\beta^i$ does not vanish, unless $s_1=\dots=s_d=0$.
In particular, we have $u^{s_1\beta+s_2\beta^2+\dots+s_d\beta^d}\neq 1$, and it follows that no nonconstant term of $f^u$ has a coefficient which lies in the subfield $\hahnconvergent{\CC}{t}$.
Yet the partial evaluation $\pi_{u\mapsto\sigma}(f)$ is defined for all $\sigma > 0$.
Moreover, $\supp(f^u)=\supp(f)$.

The following lemma describes the $u$-perturbation as a translation at the level of rank two tropical hypersurfaces.

\begin{lemma} \label{lem:perturb+structure}
  Let $f \in \hahnconvergent{\CC}{t}[x_1^{\pm},\dots,x_d^{\pm}]$ be a $d$-variate Laurent polynomial.
  Then
  \[
    \tropvariety_2(f) \ = \ \tropvariety_2(f^u) + \smallpmatrix{ 0 & 0 & \dots & 0 \\ \beta & \beta^2 & \dots & \beta^d} \enspace .
  \]
  Moreover, the same holds for the closures in the Euclidean topology, i.e.,
  \[
    \overline{\tropvariety_2(f)} \ = \ \overline{\tropvariety_2(f^u)} + \smallpmatrix{ 0 & 0 & \dots & 0 \\ \beta & \beta^2 & \dots & \beta^d} \enspace .
  \]
\end{lemma}
\begin{proof}
  Let 
  \[
  p = \smallpmatrix{
p_{11} & \dots & p_{1d} \\
p_{21} & \dots & p_{2d}
}
\in\tropvariety_2(f) \enspace .
  \]
  Then there exist distinct $s$ and $s'$ in $\supp(f)$ with $\val_2 (\gamma_s) + \langle s , p \rangle = \val_2 (\gamma_{s'}) + \langle s' , p \rangle$, where $\val_2(\gamma_s)=(\val(\gamma_s),0)$ and $\val_2(\gamma_{s'})=(\val(\gamma_{s'}),0)$.
  Hence
  \begin{equation}\label{eq:perturb+structure}
    \begin{split}
      \bigl(\val(&\gamma_s),\sum s_i\beta^i\bigr) + s_1(p_{11},p_{21}-\beta)+ \dots + s_d(p_{1d},p_{2d}-\beta^d) \\
      = \ &\val_2 (\gamma_s) + \langle s , p \rangle \ = \ \val_2 (\gamma_{s'}) + \langle s' , p \rangle \\
      = \ &\bigl(\val(\gamma_{s'}), \sum s_i\beta^i\bigr) + s'_1(p_{11},p_{21}-\beta)+ \dots + s'_d(p_{1d},p_{2d}-\beta^d) \enspace .
    \end{split}
  \end{equation}
  In other words, as $\supp(f^u)=\supp(f)$, the point 
  \[
  \smallpmatrix{
p_{11} & \dots & p_{1d} \\
p_{21} - \beta & \dots & p_{2d} - \beta^d
} = p - \smallpmatrix{ 0 & 0 & \dots & 0 \\ \beta & \beta^2 & \dots & \beta^d}
  \]
  lies in $\tropvariety_2(f^u)$, and this proves one inclusion.
  The argument can be reversed, and the claim on $\tropvariety_2(f)$ follows.
  The explicit computation in \eqref{eq:perturb+structure} carries over to the topological closure by continuity of the arithmetic operations.
\end{proof}

We recall the following concepts from \cite[\S3.6]{Tropical+Book}.
Let $f$ and $g$ be Laurent polynomials in $\hahnconvergent{\CC}{t}[x_1^{\pm},\dots,x_d^{\pm}]$.
The (polyhedral) \emph{stable intersection} of their tropical hypersurfaces is the polyhedral complex
\begin{equation}
  \tropvariety(f) \stableintersection \tropvariety(g) \ = \ \bigcup_{\dim(P + Q) = d} (P \cap Q)
\label{eq:stable+intersection}
\end{equation}
where $P$ and $Q$ are cells of $\tropvariety(f)$ and $\tropvariety(g)$, respectively.
This is a coarser notion than stable intersection of tropical varieties as it does not remember the multiplicities of the varieties.
Unless explicitly stated, we restrict purely to polyhedral stable intersection from now on.

\begin{theorem} \label{thm:stable+intersection}
  Let $f,g \in \hahnconvergent{\CC}{t}[x_1^{\pm},\dots,x_d^{\pm}]$.
  The stable intersection of $\tropvariety(f)$ and $\tropvariety(g)$ is given by projecting the set theoretic intersection of (the closures of) the rank two tropical hypersurfaces $\overline{\tropvariety_2(f)}$ and $\overline{\tropvariety_2(g^u)}$; more precisely,
  \[
    \tropvariety(f) \stableintersection \tropvariety(g) \ = \ \pi_{u*}\bigl(\overline{\tropvariety_2(f)} \cap \overline{\tropvariety_2(g^u)}\bigr) \enspace .
  \]
\end{theorem}

\begin{proof}
  Let $p_1 \in \tropvariety(f) \stableintersection \tropvariety(g) \subset \RR_t^d$.
  Then there are maximal support cells $P_S$ and $P_T$ of $\tropvariety(f)$ and $\tropvariety(g)$, respectively, containing $p_1$ with $\dim (P_S+P_T)=d$.
  Corollary \ref{cor:KK+hypersurface+closure} says that $P_S \times L_S$ and $P_T \times L_T$ are maximal polyhedra in $\overline{\tropvariety_2(f)}$ and $\overline{\tropvariety_2(g)}$, respectively.
  We have 
  \[
  \overline{\tropvariety_2(g)}=\overline{\tropvariety_2(g^u)}+\smallpmatrix{ 0 & 0 & \dots & 0 \\ \beta & \beta^2 & \dots & \beta^d}
  \]
  by Lemma \ref{lem:perturb+structure}.
  From $\dim (P_S+P_T)=d$, we infer $L_S+L_T=\RR_u^d$.
  Thus there are $q_S\in L_S$ and $q_T\in L_T$ with $q_T-q_S=(\beta,\dots,\beta^d)$.
  Hence, setting $p_2 := q_S = q_T-(\beta,\dots,\beta^d)$ and $p := p_1 + p_2$, yields
  \[
    p \in (P_S \times L_S) \cap (P_T \times (L_T - (\beta,\dots,\beta^d))) \enspace ,
  \]
  which is contained in $\overline{\tropvariety_2(f)} \cap \overline{\tropvariety_2(g^u)}$, and $\pi_{u*}(p)=p_1$.

  Conversely let $p \in \overline{\tropvariety_2(f)} \cap \overline{\tropvariety_2(g^u)} \subset \RR^{2\times d}$.
  Then there are maximal support cells $P_S$ and $P_T$ of $\tropvariety(f)$ and $\tropvariety(g)$, respectively, such that $\pi_{u*}(p)\in P_S\cap P_T$ and $\pi_{t*}(p)\in  L_S \cap (L_T - (\beta,\dots,\beta^d))$.
  We need to show that $\dim(P_S+P_T)=d$.
  As $P_S$ and $P_T$ are both maximal, we have $\dim P_S=\dim L_S=\dim L_T=\dim P_T=d-1$.
  Suppose that $\dim(P_S+P_T)<d$.
  Then $\dim(P_S+P_T)=d-1$, and the linear subspaces $L_S=L_T$ must be equal.
  As a consequence the linear subspace $L_S$ and the parallel affine subspace $L_T - (\beta,\dots,\beta^d)$ are disjoint.
  Yet this contradicts that $\pi_{t*}(p)$ lies in their intersection.
  We conclude that $\dim(P_S+P_T)=d$, and $\pi_{u*}(p)$ is contained in the stable intersection.
\end{proof}

The stable intersection of $\tropvariety(f)$ and $\tropvariety(g)$ can also be obtained by perturbing $\tropvariety(g)$ generically and taking the limit of its intersection with $\tropvariety(f)$ \cite[Proposition 3.6.12]{Tropical+Book}, i.e.,
\begin{equation}\label{eq:stable+limit}
  \tropvariety(f) \stableintersection \tropvariety(g) \ = \ \lim_{\epsilon \rightarrow 0} \bigl( \tropvariety(f) \cap (\tropvariety(g) + \epsilon {v}) \bigr)
\end{equation}
for ${v} \in \RR^d$ generic.
In this way, Theorem~\ref{thm:stable+intersection} can be seen as a version of \eqref{eq:stable+limit} based on the ``symbolic perturbation'' paradigm common in computational geometry; e.g., see \cite{Edelsbrunner87} and \cite{EmirisCannySeidel97}.

\begin{example}
  Consider the two bivariate polynomials
  \[
    f = xy + x + y + 1 \quad \text{and} \quad g = x + ty + t
  \]
  with coefficients in $\hahnfrac{\CC}{t}$.
  The intersection of their corresponding rank one tropical hypersurfaces is a ray and a point
  \[
    \tropvariety(f) \cap \tropvariety(g) \ = \ \SetOf{(\lambda+1,0)}{\lambda \geq 0} \cup \{(0,-1)\} \enspace .
  \]
  That is, the intersection at the origin is not \emph{transverse} in the sense of \cite[Definition 3.4.9]{Tropical+Book}.
  We consider $f$ and $g$ as polynomials with coefficients in $\hahnfrac{\CC}{t,u}$.
  The $u$-perturbation of $g$ is
  \[
    g^u \ = \ u^{\beta}x + tu^{\beta^2}y + t \enspace .
  \]
  The closure of their rank two tropical hypersurfaces in $\RR^{2\times 2}$ read as follows:
  \[
    \begin{aligned}
      \overline{\tropvariety_2(f)} \ =& \ \SetOf{\smallpmatrix{ \lambda_1 & 0 \\ \lambda_2 & 0 }}{\lambda_1 \geq 0,\lambda_2 \in \RR} \, \cup \, \SetOf{\smallpmatrix{ \lambda_1 & 0 \\ \lambda_2 & 0 }}{\lambda_1 \leq 0,\lambda_2 \in \RR} \\
      &\cup \, \SetOf{\smallpmatrix{ 0 & \lambda_1 \\ 0 & \lambda_2 }}{\lambda_1 \geq 0,\lambda_2 \in \RR} \, \cup \, \SetOf{\smallpmatrix{ 0 & \lambda_1 \\ 0 & \lambda_2 }}{\lambda_1 \leq 0,\lambda_2 \in \RR} \\
      \overline{\tropvariety_2(g^u)} \ =& \ \SetOf{\smallpmatrix{ 1+ \lambda_1 & 0 \\ \lambda_2 & -\beta^2 }}{\lambda_1 \geq 0, \lambda_2 \in \RR} \, \cup \, \SetOf{\smallpmatrix{ 1 & \lambda_1 \\ -\beta & \lambda_2 }}{\lambda_1 \geq 0, \lambda_2 \in \RR} \\
      &\cup \, \SetOf{\smallpmatrix{ 1 - \lambda_1 & -\lambda_1 \\ -\beta + \lambda_2 & -\beta^2 + \lambda_2 }}{\lambda_1 \geq 0, \lambda_2 \in \RR} \enspace .
    \end{aligned}
  \]
  Their intersection is the three points 
  \[
  \overline{\tropvariety_2(f)} \cap \overline{\tropvariety_2(g^u)} = \left\{\smallpmatrix{ 1 & 0 \\ -\beta & 0 } , \smallpmatrix{ 1 & 0 \\ \beta^2-\beta & 0 } , \smallpmatrix{ 0 & -1 \\ 0 & \beta - \beta^2 }\right\} \enspace .
  \]
  Projecting them via $\pi_{u*}$ yields $(1,0)$ and $(0,-1)$ in $\RR^2$.
  These two points form the stable intersection of $\tropvariety(f)$ and $\tropvariety(g)$.
\end{example}

\section{Rank two tropical convexity}
\label{sec:convex}
\noindent
Now we switch back to formal Hahn series with real coefficients.
The map
$
  \val_2 : \hahnseries{\RR}{t,u}\setminus\{0\} \to \TT_2
$
is a rank two valuation which is surjective.
It sends an element $\gamma(t,u)$ of $\hahnseries{\RR}{t,u}$ to its smallest exponent vector.
The restriction to positive series is an order reversing homomorphism of ordered semirings onto $\TT_2$, which is equipped with the lexicographic ordering; cf.\ \eqref{eq:frac+diagram}.
For instance, we have the following strict inequalities
\[
  t^9 \ < \ t^2 \ < \ tu^{1000} \ < \ tu
\]
of positive monomials, and these are equivalent to the reverse inequalities
\[
  (9,0) \ > \ (2,0) \ > \ (1,1000) \ > \ (1,1)
\]
of the exponents.
An example involving more general series, which are not necessarily positive, is
\[
  \begin{split}
    \val_2(t^9-3t^{10})=(9,0) \ &> \ \val_2(-t^2+5t^4u^2+t^{17})=(2,0) \\ &> \ \val_2(tu^{1000})=(1,1000) \ > \ \val_2(tu)=(1,1) \enspace .
  \end{split}
\]
It is useful to extend $\TT_2$ by the additional element $\bm\infty$ which is neutral with respect to the tropical addition $\min$, absorbing with respect to the tropical multiplication $+$ and larger than any element in $\TT_2$.
By letting $\val_2(0)=\bm\infty$ this yields an extension of the rank two valuation map.
This is continuous with respect to the respective order topologies.
Recall that the order topology on $\TT_2$, which agrees with $\RR^2$ as a set, is finer than the Euclidean topology.

In the subfield $\hahnseries{\RR}{u}$ we have the inequalities $0 < u < c$ for any real number $c$, and we write this as $0 < u \ll 1$.
By the same token we have
\begin{equation}\label{eq:tu}
  0 \ < \ t \ \ll \ u \ \ll \ 1
\end{equation}
in $\hahnseries{\RR}{t,u}$. 
Since our valuation prefers terms of \emph{minimal} order we say that the indeterminate $t$ \emph{dominates} $u$.

The purpose of this section is to study the interplay between three notions of convexity: \emph{ordinary convexity} with respect to the ordered field $\hahnseries{\RR}{t,u}^d$, \emph{rank two tropical convexity} with respect to tropical semifield $\TT_2$, and \emph{lex-convexity} with respect to the lexicographic ordering on $\TT_2$.
An \emph{(ordinary) cone} in $\hahnseries{\RR}{t,u}^d$ is a nonempty subset $K$ which satisfies $\lambda p + \mu q\in K$ for all $p,q\in K$ and $\lambda,\mu\geq 0$. 
It is \emph{polyhedral} if it is finitely generated.
By definition a cone in $\hahnseries{\RR}{t,u}^d$ is exactly the same as a submodule with respect to the semiring $\hahnseries{\RR}{t,u}_{\geq 0}$ of nonnegative elements.
We now make use of the notation \enquote{$\oplus$} instead of \enquote{$\min$} and \enquote{$\odot$} instead of \enquote{$+$} to stress the connection between tropical and ordinary linear algebra .
\begin{definition}
  A \emph{rank two tropical cone} in $(\TT_2\cup\{\bm\infty\})^d$ is a nonempty subset $M$ which satisfies
  \[
    (\lambda \odot p) \oplus (\mu \odot q) \ = \ \min(\lambda+p,\mu+q) \ \in \ M
  \]
  for all $p,q\in M$ and $\lambda,\mu\in\TT_2\cup\{\bm\infty\}$.
  A rank two tropical cone is \emph{polyhedral} if it is finitely generated.
\end{definition}
The following is a rank two analogue of a result by Develin and Yu \cite[Proposition 2.1]{DevelinYu07}; see also \cite[Proposition~5.8]{ETC}.
\begin{proposition}\label{prop:cone}
  Let $K$ be an ordinary cone in $\hahnseries{\RR}{t,u}_{\geq 0}^d$.
  Then $\val_2(K)$ is a rank two tropical cone in $(\TT_2\cup\{\bm\infty\})^d$, and conversely each rank two tropical cone arises in this way.
  Furthermore, if $K$ is polyhedral then $\val_2(K)$ is also, and conversely each rank two tropical polyhedral cone is the image of a polyhedral cone in the valuation map.
\end{proposition}

\begin{proof}
  As $\val_2$ is a homomorphism of semirings if restricted to positive Hahn series it follows that $\val_2(K)$ is a rank two tropical cone.
  Another consequence of this is that if $K$ is polyhedral then $\val_2(K)$ is also.

  It remains to show that, for a rank two tropical cone $M$ in $(\TT_2\cup\{\bm\infty\})^d$, there is a cone $K$ in $\hahnseries{\RR}{t,u}_{\geq 0}^d$ with $\val_2(K)=M$.
  We set $K$ to be the cone with generators
  \[
   \SetOf{(t^{p_{11}}u^{p_{21}},\dots,t^{p_{1d}}u^{p_{2d}})}{\smallpmatrix{
p_{11} & \dots & p_{1d} \\
p_{21} & \dots & p_{2d}
}\in M} \enspace ,
  \]
  where we use the convention $t^au^b=0$ for $(a,b)=\bm\infty$.
  Note that $\val_2(K)$ is a rank two tropical cone that contains $M$.
  Furthermore, as $\val_2:\hahnseries{\RR}{t,u}_{\geq0}\to\TT_2\cup\{\bm\infty\}$ is a homomorphism of semirings, any element of $\val_2(K)$ is a tropical conic combination of points in $M$, therefore $\val_2(K) = M$.
  As a further consequence of the homomorphism, if $M$ is polyhedral then $K$ must be also.
\end{proof}

A subset $K$ of $\hahnseries{\RR}{t,u}^d$ is \emph{(ordinary) convex} if $\lambda p + \mu q\in K$ for all $p,q\in K$ and $\lambda,\mu\geq 0$ with $\lambda+\mu=1$.
It is an \emph{(ordinary) polytope} if it is finitely generated.
\begin{definition}
  A subset $M$ of $(\TT_2\cup\{\bm\infty\})^d$ is \emph{rank two tropically convex} if $(\lambda \odot p) \oplus (\mu \odot q)\in M$ for all $p,q\in M$ and $\lambda,\mu\in\TT_2\cup\{\bm\infty\}$ with $\lambda\oplus\mu=(0,0)$.
  It is a \emph{rank two tropical polytope} if it is finitely generated.
\end{definition}

\begin{corollary} \label{cor:polytope}
  Let $K$ be a convex set in the positive orthant $\hahnseries{\RR}{t,u}_{\geq 0}^d$.
  Then $\val_2(K)$ is a rank two tropically convex set in $(\TT_2\cup\{\bm\infty\})^d$, and conversely each rank two tropically convex set arises in this way.
  Furthermore, if $K$ is an ordinary polytope then $\val_2(K)$ is a rank two tropical polytope, and conversely every rank two tropical polytope is the image of a polytope in the valuation map.
\end{corollary}

\begin{proof}
  All the claims follow from Proposition~\ref{prop:cone} by homogenisation.
  Indeed, consider the cone $K'$ generated by the vectors $(1,p)\in\hahnseries{\RR}{t,u}_{\geq 0}^{d+1}$ for $p\in K$.
  Then $\val_2(K')$ is a rank two tropical cone.
  The set $M$ of points $q\in(\TT_2\cup\{\bm\infty\})^d$ such that $((0,0),q)\in\val_2(K')$ is rank two tropically convex and $\val_2(K)=M$.
\end{proof}

All of the above can be generalised to other valued fields of arbitrary rank with surjective valuation.

\paragraph*{Convergent real Hahn series.}
Now we consider a convergent subring $\hahnconvergent{\RR}{t,u}$ of $\hahnseries{\RR}{t,u}$ whose valuation map is surjective.
Then we can combine the diagram \eqref{eq:frac+diagram} with Proposition~\ref{prop:cone} to get a third diagram, this time of modules over semirings, i.e., cones.
As before $\pi_{u\mapsto\sigma}$ does not globally commute and depends on the choice of $\sigma$.

\begin{equation}\label{eq:convex+diagram}
  \begin{tikzcd}
    \hahnconvergent{\RR}{t}_{\geq0}^d \arrow{d}{\val} \arrow{r}{\iota} &  \hahnconvergent{\RR}{t,u}_{\geq0}^d \arrow{d}{\val_{2}} \arrow[dashed]{r}{\pi_{u\mapsto\sigma}} & \hahnconvergent{\RR}{t}_{\geq0}^d \arrow{d}{\val} \\
    (\TT\cup\{\infty\})^d \arrow{r}{\iota_*} & (\TT_{2}\cup\{\bm\infty\})^d \arrow{r}{\pi_{u*}} & (\TT\cup\{\infty\})^d
  \end{tikzcd}
\end{equation}

As the tropicalisation of any ordinary cone or polytope in $\hahnconvergent{\RR}{t,u}_{\geq 0}^d$ is a rank two tropical cone or polytope, any results on the latter objects hold also for the former.
Additionally, any results for rank two tropical cones give analogous results for rank two tropical polytopes by homogenisation.
Therefore for simplicity, we shall state the results for rank two tropical cones only, and we use rank two tropical polytopes in the examples.

The connection between ordinary and tropical polytopes (in rank one) is rather loose.
Often it is difficult to carry over combinatorial information.
For example, if an ordinary polytope is not sufficiently generic, distinct faces may have the same image in the valuation map.
This is a disadvantage, as some algorithms rely on that information.
Specifically, the tropical simplex method of \cite{ABGJ:2015}, which solves tropical linear programs, requires that the input is generic enough.
To obtain \cite[Theorem~4.3]{ABGJ:2014}, a result on algorithmic complexity, that obstacle was overcome in \cite[Theorem~3.12]{ABGJ:2014} via tropical convexity of higher rank.
The following example exhibits the idea.
\begin{example}\label{exmp:pyramid}
  We consider the ordinary polyhedron $\cP$ in the $3$-dimensional vector space over $\hahnconvergent{\RR}{t,u}$ given by the linear inequalities
  \begin{equation}\label{eq:pyramid:ordinary}
    \begin{array}{rclcrcl}
      (1+u)x_1 & \geq & t^2x_3 &\quad& 2t & \geq & 2x_2 + tx_3 \\
      2t^2 + (2-2t) x_2 & \geq & 2x_1 + (t+t^2)x_3 && 4x_2 & \geq & 2tx_3\\
      x_3 &\geq& 0  && && 
    \end{array}
  \end{equation}
  In fact, $\cP$ lies in the positive orthant.  Its rank two tropicalisation is given by the tropical linear inequalities
  \begin{equation}\label{eq:pyramid:tropical}
    \begin{array}{rclcrcl}
      (0,1)x_1 &\!\leq\!& (2,0)x_3 &\hspace*{-0.5cm}& (1,0) & \!\leq\! & \min \{x_2 , (1,0)x_3 \} \\
      \min \{(2,0) , x_2 \} & \!\leq\! & \min \{ x_1, (1,0)x_3 \} &\hspace*{-0.5cm}& x_2 & \!\leq\! & (1,0)x_3
    \end{array}
  \end{equation}
  In \eqref{eq:pyramid:tropical} we omit the tropical multiplication symbol \enquote{$+$}, and the tropical nonnegativity constraint $x_3\leq\bm\infty$ is implicit.

  Letting $u=0$ in \eqref{eq:pyramid:ordinary} yields an ordinary polyhedron $\cP_0$ over $\hahnconvergent{\RR}{t}$, and this is combinatorially equivalent to a pyramid with quadrangular base.
  The rank one tropicalisation of $\cP_0$ is given by $A\odot x \oplus b \leq A' \odot x \oplus b'$ where
  \[
    A=\begin{pmatrix}
      0 & \infty & \infty \\
      \infty & \infty & \infty \\
      \infty & 0 & \infty \\
      \infty & 0 & \infty \\
    \end{pmatrix} \,,\
    b=\begin{pmatrix}
      \infty \\
      1 \\
      2 \\
      \infty
    \end{pmatrix} \,,\
    A'=\begin{pmatrix}
      \infty & \infty & 2 \\
      \infty & 0 & 1 \\
      0 & \infty & 1 \\
      \infty & \infty & 1 \\
    \end{pmatrix} \,,\
    b'=\begin{pmatrix}
      \infty \\
      \infty \\
      \infty \\
      \infty
    \end{pmatrix}
  \]
  The pair of extended tropical matrices $(A \ b)$, $(A' \ b')$ is tropically sign singular i.e., $\val(\cP_0)$ is not sufficiently generic.
  As a consequence the tropical simplex algorithm from \cite{ABGJ:2015} cannot be applied directly to optimise some tropical linear objective function over $\val(\cP_0)$.

  The rank two tropical polyhedron $\val_2(\cP)$ arises from $\val(\cP_0)$ via an infinitesimal perturbation, similar to the higher rank interpretation of stable intersection from Section~\ref{sec:stable}.
  By \eqref{eq:convex+diagram} this commutes with a perturbation of $\cP_0$ to $\cP$.
  The rank two lift $\cP$ is a simple polytope, and thus its combinatorics is entirely encoded in its vertex-edge graph \cite[\S3.4]{Ziegler:Lectures+on+polytopes}.
  That property is preserved for $\cP_\sigma=\pi_{u\mapsto\sigma}(\cP)$, where $\sigma$ is a sufficiently small positive real number.
  In this way, combining \cite{ABGJ:2015} with a lift to Hahn series of higher rank yields the tropical simplex algorithm for degenerate input from \cite{ABGJ:2014}.
\end{example}

\paragraph*{Lex-polyhedral decompositions.}
Rank one tropical cones have an explicit description as a polyhedral complex in terms of their \emph{covector decomposition}; see \cite[\S5.2]{Tropical+Book} and \cite[\S6.3]{ETC}.
As with rank two tropical hypersurfaces, rank two tropical cones are not closed in the Euclidean topology; cf.\ Figure~\ref{fig:interval_image}, therefore they do not have a polyhedral decomposition in the ordinary sense.
However, we can construct an analogous decomposition in terms of lex-polyhedra by building on the corresponding notions in rank one.

\begin{figure}
  \includegraphics[width=0.6\columnwidth]{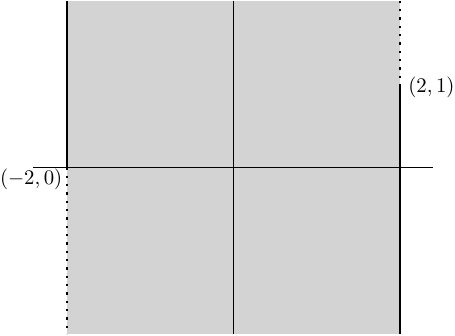}
  \caption{The tropicalisation of the ordinary interval $[t^2u,t^{-2}]$ in $\hahnconvergent{\RR}{t,u}$ as a subset of $\TT_2$.
  It is a tropically convex set generated by $\{(-2,0),(2,1)\}$.
  Note that it is not closed under the Euclidean topology as the dotted boundary is not part of the interval.}
  \label{fig:interval_image}
\end{figure}

Given a point $u \in (\TT_2\cup\{\bm\infty\})^d$ with $u_i \neq \bm\infty$, we define its \emph{$i$th sector}
\[
\lex{Z}_i(u) \ = \ \bigcap_{k \in [d], u_k \neq \bm\infty}\SetOf{p \in \TT_2^d}{p_k-p_i \leq u_k-u_i} \ = \ \bigcap_{k \in [d], u_k \neq \bm\infty} \lex{H}_{e_k-e_i,u_k - u_i}
\]
where $e_1,\dots,e_d \in \ZZ^d$ are the standard unit vectors.
Observe that by definition each sector is a lex-polyhedron.

\begin{remark} \label{rem:max+sector}
As the two operations behave isomorphically, one can choose tropical addition to be $\min$ or $\max$.
The rank $m$ tropical max-plus semiring $\TT_m^{\max} = (\RR^m, \max, +)$ is appended with the additive identity element $-\bm\infty$, the smallest element under the lexicographical ordering.
This allows us to give some geometric intuition to the sectors $\lex{Z}_i(u)$.

Given a point $u \in (\TT_2 \cup \{\bm\infty\})^d$, consider the max-tropical linear form $F_u = \max\smallSetOf{x_i - u_i}{i \in [d] \ , \ u_i \neq \bm\infty}$.
Its support is the set of standard unit vectors $\supp(F_u) = \smallSetOf{e_i}{i \in [d] \ , \ u_i \neq \bm\infty}$.
As with min-tropical hypersurfaces, its max-tropical hypersurface $\tropvariety_2(F_u)$ is the locus of points at which $F_u$ is non-linear.
The results of Section \ref{sec:hypersurfaces} hold for $\tropvariety_2(F_u)$, in particular it induces a decomposition of $\TT_2^d$ in terms of support cells.
Comparing definitions implies the sector $\lex{Z}_i(u)$ is the precisely the set of points in the support cell $\lex{P}_{e_i}$ induced by $\tropvariety_2(F_u)$.
Furthermore, these sectors can be considered as translated lex-cones, where a \emph{lex-cone} is the intersection of linear lex-halfspaces.
Therefore the lex-polyhedral cell complex $\tropvariety_2(F_u)$ induced is a translated lex-polyhedral fan (i.e., it consists of translated lex-cones) whose apex is the point $u$.
\end{remark}

In the sequel let $\cK$ be a rank two tropical cone, equipped with a fixed system of (labelled) generators $V=(v^{(1)},\dots,v^{(n)})$, where $v^{(j)}\in(\TT_2\cup\{\bm\infty\})^d$.
\begin{lemma} \label{lem:cone+cover}
  A point $p \in (\TT_2)^d$ is contained in $\cK$ if and only if for each $i \in [d]$, there exists some $j \in [n]$ such that $p \in \lex{Z}_i(v^{(j)})$.
\end{lemma}
\begin{proof}
  The proof of \cite[Proposition 5.37]{ETC} generalises directly.
\end{proof}

As in \cite[\S3.2]{JoswigLoho:2016} and \cite[\S6.3]{ETC}, Lemma \ref{lem:cone+cover} inspires the following combinatorial data.
Given a point $p \in (\TT_2)^d$, we define its \emph{covector} $S_p=S_p(V)$ to be the bipartite graph on the node set $[d] \sqcup [n]$ where $(i,j) \in S_p$ if and only if $p \in \lex{Z}_i(v^{(j)})$.
We say a covector is \emph{bounded} if no node in $[d]$ is isolated.
With this, we can restate Lemma \ref{lem:cone+cover} as $p \in \cK$ if and only if $S_p$ is bounded.
By definition, the points with a given covector $S$ satisfy the inequalities
\begin{equation} \label{eq:covector+inequalities}
  p_k - p_i \leq v_k^{(j)} - v_i^{(j)} \quad \text{for all } k \in \supp(v^{(j)}) \text{ where } (i,j) \in S \enspace .
\end{equation}
Note that these hold also for any point whose covector contains $S$.
We define the \emph{covector cell}
\[
  \lex{C}_S(V) \ = \ \SetOf{p \in (\TT_2)^d}{S \subseteq S_p} \enspace ,
\]
and immediately note that $\lex{C}_S=\lex{C}_S(V)$ is a lex-polyhedron, as it is cut out by lex-halfspaces defined by the family of inequalities \eqref{eq:covector+inequalities}.
As with support cells, there may be bipartite graphs $S,T$ such that $\lex{C}_S = \lex{C}_T$, but there is always a maximal bipartite graph defining the cell; this is the covector.

\begin{lemma}
  The covector cell $\lex{C}_S$ is rank two tropically convex.
\end{lemma}
\begin{proof}
  Let $p$ and $q$ be points in $\lex{C}_S(V)$.
  It suffices to show that for $\mu\in\TT_2$ with $\mu \geq(0,0)$ we have $p \oplus (\mu \odot q) \in \lex{C}_S(V)$.
  This follows from
  \[
    \begin{split}
      (p_k&\oplus (\mu\odot q_k)) - (p_i \oplus (\mu\odot q_i) \ = \ \min(p_k,\mu+q_k) - \min(p_i,\mu+q_i) \\
      &= \ \min(p_k-p_i,p_k-\mu-q_i,\mu+q_k-p_i,q_k-q_i) \\
      &\leq \ \min(p_k-p_i, q_k-q_i) \ \leq \ v_k^{(j)} - v_i^{(j)} \quad \text{for all } k \in \supp(v^{(j)}) \enspace .
    \end{split}
  \] \\[-2em] 
\end{proof}

This means that the covector cells $\lex{C}_S$ are both lex-polyhedra and rank two tropically convex; i.e., they form rank two analogues of the polytropes in \cite[\S6.5]{ETC}.
Covector cells $\lex{C}_S$ have some further nice combinatorial properties, analogous to support cells:
\begin{lemma} \label{lem:covector+cells}
  Let $S,T$ be bipartite graphs on $[d] \sqcup [n]$ such that no node of $[n]$ is isolated.
  \begin{enumerate}
  \item $\lex{C}_S \cap \lex{C}_T = \lex{C}_{S\cup T}$.
  \item $S \subseteq  T$ if and only if $\lex{C}_T$ is a face of $\lex{C}_S$.
  \end{enumerate}
\end{lemma}
\begin{proof}
  Both claims are immediate generalisations of \cite[Observation 6.10]{ETC}.
\end{proof}
The second statement of Lemma \ref{lem:covector+cells} implies that given a covector cell $\lex{C}_S$, its relative interior, denoted $\inte(\lex{C}_S)$, is the set of points whose covector is precisely $S$.
We recall that as $\lex{C}_S$ is a lex-polyhedron, $\inte(\lex{C}_S)$ is open in the order topology but not in the Euclidean topology.
The following generalises the covector decomposition in rank one from \cite[\S6.3]{ETC} and \cite[\S5.2]{Tropical+Book}.

\begin{theorem}\label{thm:lex-complex}
  The intersection $\cK\cap(\TT_2)^d$ decomposes as a lex-polyhedral complex whose cells are of the form $\lex{C}_S$ where $S$ is a bounded covector with respect to the generating system~$V$.
\end{theorem}
\begin{proof}
  Lemma \ref{lem:cone+cover} shows that the collection of lex-polyhedra 
  \[
    \lex{\Sigma} \ = \ \SetOf{\lex{C}_S}{S \text{ bounded covector}}
  \] covers $\cK\cap(\TT_2)^d$.
  Lemma \ref{lem:covector+cells} shows that $\lex{\Sigma}$ is closed under intersections and taking faces, and therefore is a lex-polyhedral complex.
\end{proof}

\begin{remark} \label{rem:support+covector}
Recall from Remark \ref{rem:max+sector} that the rank two max-tropical hyperplane $\tropvariety_2(F_u)$ induces a decomposition of $\TT_2^d$ into a lex-polyhedral fan.
Furthermore, the maximal lex-cones are the sectors $\lex{Z}_i(u)$ equal to the support cell $\lex{P}_{e_i}$.
Given the generating set $V = \{v^{(1)},\dots,v^{(n)}\}$, the covector cell $\lex{C}_S$ is equal to the finite intersection
\[
\lex{C}_S \ = \ \bigcap_{(i,j) \in S} \lex{Z}_i(v^{(j)}) \enspace .
\]
Therefore the covector decomposition is precisely the common refinement of the lexicographical fan structures induced by the max-tropical hyperplanes $\tropvariety_2(F_{v^{(j)}})$.
Moreover, taking the product of the max-tropical linear forms gives the rank two max-tropical multilinear form $F_V = \bigodot F_{v^{(j)}}$.
The support sets of $F_V$ are precisely the covectors induced by $V$, implying covectors are a special case of support sets.
This generalises \cite[Corollary~6.16]{ETC}.
\end{remark}

For a rank two tropical cone $\cK$ generated by $V = \{v^{(1)}, \dots, v^{(n)}\}$ and a covector $T$, we let $\cK_T$ denote the rank two tropical cone generated by $V_T = \{v_T^{(1)},\dots,v_T^{(n)}\}$ where
\[
(v^{(j)}_T)_i \ = \
\begin{cases}
v^{(j)}_i & \text{if } (i,j) \in T \\
\bm\infty & \text{otherwise \enspace .}
\end{cases}
\]
Similar to support cells, we denote the covector cell of $\cK_T$ with covector $S$ as $B_{S,T}$.
The following results give decompositions for rank two tropical cones in terms of the interiors of polyhedra and ordinary polyhedra, analogous to Theorem \ref{thm:union+polyhedra} and Corollary \ref{cor:hypersurface+closure}.

\begin{theorem} \label{thm:cone+structure}
  Let $\cK$ be a rank two tropical cone generated by $V = \{v^{(1)},\dots,v^{(n)}\} \subset (\TT_2 \cup \{\bm\infty\})^d$.
  The intersection $\cK \cap (\TT_2)^d$ is the finite disjoint union
  \[
    \cK \cap \TT_2^d \ = \ \bigsqcup_S \bigsqcup_{T \supseteq S} \bigl(\inte(A_T) \times \inte(B_{S,T}) \bigr)
  \]
  of interiors of polyhedra in $\RR^{2\times d}$, where $A_T$ and $B_{S,T}$ are covector cells of the rank one tropical cones $\pi_{u*}(\cK)$ in $\RR_t^d$ and $\pi_{t*}(\cK_T)$ in $\RR_u^d$ respectively.
\end{theorem}
\begin{proof}
By Theorem \ref{thm:lex-complex}, $\cK \cap \TT_2^d$ is the union of lex-polyhedral cells $\lex{C}_S$ as $S$ runs over all covectors.
Furthermore, the second statement of Lemma \ref{lem:covector+cells} implies this union becomes disjoint if we restrict to the interiors of $\lex{C}_S$.
Note that each $\inte(\lex{C}_S)$ is a lex-open polyhedron.
We claim that $\inte(\lex{C}_S) = \bigsqcup_{T \supseteq S} \bigl(\inte(A_T) \times \inte(B_{S,T}) \bigr)$.

The point $p$ is contained in $\inte(\lex{C}_S)$ if and only if for each $v^{(j)}$:
\begin{equation} \label{eq:S+tight+cone}
p_k - v_k^{(j)} \leq p_i -v_i^{(j)} \quad \text{for all } k \in \supp(v^{(j)}) \text{ where } (i,j) \in S \enspace .
\end{equation}
with equality if and only if $(k,j) \in S$.
Considering the lexicographical ordering on $\TT_2$ and its coordinates separately, this is equivalent to the following two conditions:
\begin{align} \label{eq:tcone}
\pi_{u*}(p_k) - \pi_{u*}(v_k^{(j)}) \leq \pi_{u*}(p_i) - \pi_{u*}(v_i^{(j)}) \enspace ,
\end{align}
for all $k \in \supp(v^{(j)})$ and $(i,j) \in T$ for some $T \supseteq S$, with equality if and only if $(k,j) \in T$.
\begin{align} \label{eq:ucone}
\pi_{t*}(p_k) - \pi_{t*}((v_T^{(j)})_k) \leq \pi_{t*}(p_i) - \pi_{u*}((v_T^{(j)})_i) \enspace ,
\end{align}
for all $k \in \supp(v_T^{(j)})$ and $(i,j) \in S$, with equality if and only if $(k,j) \in S$.
Condition \eqref{eq:tcone} is equivalent to $\pi_{u*}(p)$ being contained in the relative interior of the covector cell $A_T$ of $\pi_{u*}(\cK)$.
Condition \eqref{eq:ucone} is equivalent to $\pi_{t*}(p)$ being contained in the relative interior of the covector cell $B_{S,T}$ of $\pi_{t*}(\cK_T)$.

It remains to show each part of the disjoint union is the interior of a polyhedron.
The proof is identical to the end of the proof of Theorem \ref{thm:union+polyhedra}.
\end{proof}

\begin{corollary} \label{cor:cone+closure}
  With the notation of Theorem~\ref{thm:cone+structure}: the closure of $\cK \cap \TT_2^d$ in the Euclidean topology is the finite union
  \[
    \overline{\cK \cap \TT_2^d} \ = \ \bigcup_S \bigcup_{T \supseteq S} \bigl(A_T \times B_{S,T} \bigr)
  \]
  of polyhedra in $\RR^{2\times d}$.
\end{corollary}
\begin{proof}
  As $A_T \times B_{S,T} = \overline{\inte(A_T) \times \inte(B_{S,T})}$, the result follows from Theorem \ref{thm:cone+structure} and that the closure of a finite union of sets equals the union of their closures.
\end{proof}

Recall that Diagram \eqref{eq:convex+diagram} says $\pi_{u*}$ and $\pi_{u\mapsto\sigma}$ (and $\pi_{t*}$ and $\pi_{t\mapsto\rho}$) commute with the valuation map.
Therefore if $\cK = \val_2(K)$ for some ordinary cone $K \subset \hahnconvergent{\RR}{t,u}^d$, we can obtain an analogous result to Theorem \ref{thm:cone+structure} in terms of the covector decompositions of $\val(\pi_{u\mapsto\sigma}(K))$ and $\val(\pi_{t\mapsto\rho}(K))$.

As with Corollary \ref{cor:polytope}, we can obtain an analogous statement to Theorem~\ref{thm:cone+structure} and Corollary~\ref{cor:cone+closure} for tropical polytopes by dehomogenisation.
Explicitly, given some generating set $V \subset (\TT_2 \cup \{\bm\infty\})^d$ for a convex polytope $\cK$, we can consider the cone $\cK' \subset (\TT_2 \cup \{\bm\infty\})^{d+1}$ generated by
\[ \SetOf{((0,0),v^{(j)})}{v^{(j)} \in V} \enspace. \]
Then $\cK$ inherits the structure of $\cK'$ intersected with the hyperplane $\{x_0 = (0,0)\}$.
Note that Diagram \eqref{eq:convex+diagram} implies we can do this dehomogenisation in $\hahnconvergent{\RR}{t,u}_{\geq 0}^{d+1}$.

\section{Concluding remarks and open questions}
\label{sec:remarks}
\noindent
To avoid cumbersome notation in this article, we decided to restrict our exposition to rank two tropical objects.
Yet the characterisations of rank two tropical hypersurfaces and cones can be generalised to arbitrary finite rank by recursively exploiting the structure of tropical hypersurfaces and cones of corank one.
This entails a generalisation of Theorem~\ref{thm:stable+intersection} to the simultaneous stable intersection of any finite number of tropical hypersurfaces.
We leave the details to the reader.

A rank one tropical hypersurface, given by a tropical polynomial $F$, is dual to the regular subdivision of the point configuration given by the monomials of $F$, where the coefficients yield the height function; cf.\ \cite[Proposition 3.1.6]{Tropical+Book}.
\begin{question}\label{qst:regular+refinement}
  How does this generalise to higher rank?
\end{question}
This should be related to the \emph{regular refinement} of subdivisions in the sense of \cite[Definition 2.3.17]{Triangulations}.
Furthermore, our current setup for rank two tropical hypersurfaces is purely polyhedral, and so this does not capture any arithmetic properties.
\begin{question}
  What is a good notion of multiplicity for tropical hypersurfaces of higher rank?
\end{question}
In this context it could be interesting to investigate the recent work of Gwo\'{z}dziewicz, Hejmej and Schober on the factorisation of formal power series of higher rank \cite{GwozdziewiczHejmej:1807.04944}.

Proposition \ref{prop:lex+polyhedral+hypersurface} and Theorem \ref{thm:lex-complex} describe rank two objects as a lex-polyhedral complex, and moreover gives a canonical inequality description for each.
However, we only know little about lex-polyhedra.
\begin{question}\label{qst:representation}
  What can be said about the combinatorial properties of lex-polyhedra?
\end{question}
The theory of non-trivial divisible totally ordered abelian groups known to be complete; e.g., see \cite[\S2]{Hils:2018}.
As a consequence lex-polyhedra, which have integral slopes by definition, should share some properties of rational ordinary polyhedra.

\bibliographystyle{plain}
\bibliography{stable}

\end{document}